\documentclass{amsart}
\usepackage{amsmath}
\usepackage{amscd}
\usepackage{amssymb}
\usepackage{amsthm}
\usepackage{color}
\usepackage{tcolorbox}
\RequirePackage{filecontents}
\usepackage{fullpage}
\usepackage{longtable}
\usepackage[numbers]{natbib}  
\usepackage[draft]{hyperref}
\theoremstyle{plain}
\newtheorem{prop}{Proposition}
\newtheorem{thm}[prop]{Theorem}
\newtheorem{lem}[prop]{Lemma}
\newtheorem{cor}[prop]{Corollary}
\theoremstyle{definition}

\newtheorem{ex}[prop]{Example}
\newtheorem{rem}[prop]{Remark}

\newenvironment{psmallmatrix}
  {\left(\begin{smallmatrix}}
  {\end{smallmatrix}\right)}
\begin{document}

\nocite{*}

\title{Graded rings of paramodular forms of levels $5$ and $7$}
\author{Brandon Williams }

\address{Fachbereich Mathematik \\ Technische Universit\"at Darmstadt \\ 64289 Darmstadt, Germany}

\email{bwilliams@mathematik.tu-darmstadt.de}

\maketitle

\begin{abstract} We compute generators and relations for the graded rings of paramodular forms of degree two and levels 5 and 7. The generators are expressed as quotients of Gritsenko lifts and Borcherds products. The computation is made possible by a characterization of modular forms on the Humbert surfaces of discriminant 4 that arise from paramodular forms by restriction.\end{abstract}

\section{Introduction}

Paramodular forms (of degree $2$, level $N \in \mathbb{N}$, and weight $k$) are holomorphic functions $f$ on the Siegel upper half-space $\mathbb{H}_2$ which transform  under the action of the paramodular group $$K(N) = \{M \in \mathrm{Sp}_4(\mathbb{Q}): \; \sigma_N^{-1}M\sigma_N \in \mathbb{Z}^{4 \times 4}\}, \; \; \sigma_N = \mathrm{diag}(1,1,1,N)$$ by $f(M \cdot \tau) = \mathrm{det}(c \tau + d)^k f(\tau)$ for $M = \begin{psmallmatrix} a & b \\ c & d \end{psmallmatrix} \in K(N)$ and $\tau \in \mathbb{H}_2$. \\

For a fixed level $N$, paramodular forms of all integral weights form a finitely generated graded ring $M_*(K(N))$ and a natural question is to ask for the structure of this ring. This yields information about the geometry of $X_{K(N)} = \overline{K(N) \backslash \mathbb{H}_2}$ (a moduli space for abelian surfaces with a polarization of type $(1,N)$) since the Baily-Borel isomorphism identifies $X_{K(N)}$ with the projective cone $\mathrm{Proj} \, M_*(K(N))$. Unfortunately these rings are difficult to compute. Besides Igusa's celebrated result for $N=1$ \cite{Ig2}, the ring structure is only understood for levels $N = 2$ (by Ibukiyama and Onodera \cite{IO}), $N=3$ (by Dern \cite{Dern2}) and $N = 4$ (the group $K(4)$ is conjugate to a congruence subgroup of $\mathrm{Sp}_4(\mathbb{Z})$ for which this is implicit in the work of Igusa \cite{Ig2}). Substantial progress in levels $N=5,7$ was made by Marschner \cite{M} and Gehre \cite{G} respectively but the problem has remained open for all levels $N \ge 5$. \\

A general approach to these problems is to use pullback maps to lower-dimensional modular varieties and the existence of modular forms with special divisors. The levels $N = 1,2,3,4$ admit a paramodular form which vanishes only on the $K(N)$-orbit of the diagonal (by \cite{GH2}). Any other paramodular form $f$ can be evaluated along the diagonal through the Witt operator, which we denote $$P_1 : M_*(K(N)) \longrightarrow M_*(\mathrm{SL}_2(\mathbb{Z}) \times \mathrm{SL}_2(\mathbb{Z})), \; \; P_1 f(\tau_1,\tau_2) = f ( \begin{psmallmatrix} \tau_1 & 0 \\ 0 & \tau_2 / N \end{psmallmatrix}).$$ One constructs a family of paramodular forms whose images under $P_1$ generate the ring of modular forms for $\mathrm{SL}_2(\mathbb{Z}) \times \mathrm{SL}_2(\mathbb{Z})$ with appropriate characters. Any paramodular form can then be reduced against this family to yield a form which vanishes on the diagonal and is therefore divisible by the distinguished form with a simple zero by the Koecher principle. In this way the graded ring can be computed by induction on the weight. \\

Unfortunately in higher levels $N \ge 5$ it is never possible to find a paramodular form which vanishes only along the diagonal (by Proposition 1.1 of \cite{GH2}) so this argument fails. (In fact, allowing congruence subgroups hardly improves the situation; see the classification in \cite{CG}. Some related computations of graded rings were given by Aoki and Ibukiyama in \cite{AI}.) One might instead try to reduce against paramodular forms whose divisor consists not only of the diagonal but also Humbert surfaces of larger discriminant $D > 1$ (which correspond to polarized abelian surfaces with special endomorphisms; see e.g. the lecture notes \cite{Dolg} for an introduction), the diagonal being a Humbert surface of discriminant one. Such paramodular forms can be realized as Borcherds products. There are instances in the literature where this approach has succeeded (e.g. \cite{DK2}). However the pullbacks (generalizations of the Witt operator) to Humbert surfaces other than the diagonal are more complicated to work with explicitly and are usually not surjective, with the image being rather difficult to determine in general. \\

In this note we take a closer look at the pullback $P_4$ to the Humbert surfaces of discriminant four for odd prime levels $N$. In particular we list $5$ candidate modular forms which one might expect to generate the image of symmetric paramodular forms under $P_4$. They do generate it in levels $N=5,7$ and this reduces the computation of the graded ring $M_*(K(N))$ to a logical puzzle of constructing paramodular forms which vanish to varying orders along certain Humbert surfaces. (A similar argument is outlined by Marschner and Gehre in \cite{M} and \cite{G} respectively, although we do not follow their suggestion to reduce along the Humbert surface of discriminant $9$. Instead we use Humbert surfaces of discriminants $1,4,5$ and $8$.)\\

We can prove the following theorems. Let $\mathcal{E}_k$ be the paramodular Eisenstein series of weight $k$.

\begin{thm} In level $N = 5$, there are Borcherds products $b_5,b_8,b_{12},b_{14}$, Gritsenko lifts $g_6,g_7,g_8,g_{10}$, and holomorphic quotient expressions $h_9,h_{10},h_{11},h_{12},h_{16}$ in them such that the graded ring $M_*(K(5))$ is minimally presented by the generators $$\mathcal{E}_4,b_5,\mathcal{E}_6,g_6,g_7,g_8,b_8,h_9,g_{10},h_{10},h_{11},b_{12},h_{12},b_{14},h_{16}$$ of weights $4,5,6,6,7,8,8,9,10,10,11,12,12,14,16$ and by $59$ relations in weights $13$ through $32$.
\end{thm}

\begin{thm} In level $N=7$, there are Borcherds products $b_4,b_6,b_7,b_9,b_{10},b_{12}^{sym},b_{12}^{anti},b_{13}$, Gritsenko lifts $g_5,g_6,g_7,g_8,g_{10}$, and holomorphic quotient expressions $h_8,h_9,h_{11},h_{14},h_{15},h_{16}$ in them such that the graded ring $M_*(K(7))$ is minimally presented by the generators $$\mathcal{E}_4,b_4,g_5,\mathcal{E}_6,b_6,g_6,b_7,g_7,g_8,h_8,b_9,h_9,b_{10},g_{10},h_{11},b_{12}^{sym},b_{12}^{anti},b_{13},h_{14},h_{15},h_{16}$$ of weights $4,4,5,6,6,6,7,7,8,8,9,9,10,10,11,12,12,13,14,15,16$ and by $144$ relations in weights $10$ through $32$.
\end{thm}

The definitions of the forms $b_i,g_i,h_i$ are given in sections $5$ and $6$ below. The relations are listed in the ancillary files on arXiv. Fourier coefficients of the generators are available on the author's university webpage. \\

\textbf{Acknowledgments:} The computations in this note were done in Sage and Macaulay2. I also thank Jan Hendrik Bruinier and Aloys Krieg for helpful discussions. This work was supported by the LOEWE-Schwerpunkt Uniformized Structures in Arithmetic and Geometry.

\section{Notation}

$K(N)$ is the integral paramodular group of degree two and level $N$. $K(N)^+$ is the group generated by $K(N)$ and the Fricke involution $V_N$. $M_*(K(N)) = \oplus_{k=0}^{\infty} M_k(K(N))$ is the graded ring of paramodular forms. \\

$\mathbb{H} = \{\tau = x+iy: \, y > 0\}$ is the upper half-plane. $\mathbb{H}_2$ is the Siegel upper half-space of degree two; its elements are either also labeled $\tau$ or in matrix form $\begin{psmallmatrix} \tau & z \\ z & w \end{psmallmatrix}$. We write $q = e^{2\pi i \tau}$, $r = e^{2\pi i z}$, $s = e^{2\pi i w}$. When $z$ is the elliptic variable of a Jacobi form we write $\zeta = e^{2\pi i z}$. For $D \in \mathbb{N}$, $\mathcal{H}_D$ is the discriminant $D$ Humbert surface (the solutions of \emph{primitive} singular equations of discriminant $D$). \\

We denote by $G$ the group $$G = \{(M_1,M_2) \in \mathrm{SL}_2(\mathbb{Z}) \times \mathrm{SL}_2(\mathbb{Z}): \; M_1 \equiv M_2 \, (\text{mod} \, 2)\}$$ and by $A_*$ and $A_*^{sym}$ particular graded subrings of modular forms for $G$ which are defined in section 4. \\

If $r,s > 0$ then $II_{r,s}$ is the (unique up to isomorphism) even unimodular lattice of signature $(r,s)$. For an even lattice $(\Lambda,Q)$ and $N \in \mathbb{Z} \backslash \{0\}$, we denote by $\Lambda(N)$ the rescaled lattice $(\Lambda, N \cdot Q)$. \\

We use Eisenstein series for various groups. To reduce the risk of confusion we always let $E_k$ be the classical (elliptic) Eisenstein series for $\mathrm{SL}_2(\mathbb{Z})$; we let $\mathbf{E}_k$ be the Hecke Eisenstein series (a Hilbert modular form) for a real-quadratic field; and we let $\mathcal{E}_k$ be the paramodular Eisenstein series.

\section{Paramodular forms of degree two}

We continue the introduction of paramodular forms. For $N \in \mathbb{N}$, the \textbf{paramodular group} of level $N$ is the group $K(N)$ of symplectic matrices of the form $\begin{psmallmatrix} * & * & * & */N \\ * & * & * & */N \\ * & * & * & */N \\ N* & N* & N* & * \end{psmallmatrix}$ where $*$ represent integers. This acts on the upper half-space $\mathbb{H}_2$ in the usual way, i.e. for a block matrix $M = \begin{psmallmatrix} a & b \\ c & d \end{psmallmatrix} \in K(N)$ and $\tau \in \mathbb{H}_2$ we set $M \cdot \tau = (a\tau + b)(c \tau + d)^{-1}.$ \\

A \textbf{paramodular form} of weight $k \in \frac{1}{2}\mathbb{N}$ is a holomorphic function $F : \mathbb{H}_2 \rightarrow \mathbb{C}$ satisfying $F(M \cdot \tau) = \mathrm{det}(c \tau + d)^k  F(\tau)$ for all $M \in K(N)$. (The Koecher principle states that $F$ extends holomorphically to the boundary of $K(N) \backslash \mathbb{H}_2$ so we omit this from the definition.) The invariance of $F$ under the translations $T_b = \begin{psmallmatrix} 1 & 0 & b_1 & b_2 \\ 0 & 1 & b_2 & b_3/N \\ 0 & 0 & 1 & 0 \\ 0 & 0 & 0 & 1 \end{psmallmatrix},$ $b_1,b_2,b_3 \in \mathbb{Z}$ implies that $F$ is given by a Fourier series, which we write in the form $$F( \begin{psmallmatrix} \tau & z \\ z & w \end{psmallmatrix}) = \sum_{a,b,c \in \mathbb{Z}} \alpha(a,b,c) q^a r^b s^{Nc}, \; q = e^{2\pi i \tau}, \; r = e^{2\pi i z}, \; s = e^{2\pi i w}, \; \alpha(a,b,c) \in \mathbb{C},$$ and the Koecher principle can also be interpreted as the condition $\alpha(a,b,c) = 0$ unless $a,c \ge 0$ and $b^2 \le 4Nac$. \\

The paramodular group is normalized by an additional map called the \textbf{Fricke involution}: $$V_N : \begin{psmallmatrix} \tau & z \\ z & w \end{psmallmatrix} \mapsto \begin{psmallmatrix} Nw & -z \\ -z & \tau /N \end{psmallmatrix}.$$ (If $N = 1$ then this is already contained in $\mathrm{Sp}_4(\mathbb{Z})$.) An \textbf{extended paramodular form} is a paramodular form $F$ which is invariant under $V_N$, i.e. $F|V_N = F$ where $|$ is the usual slash operator. The extended paramodular group $\langle K(N), V_N \rangle$ will be denoted $K(N)^+$. (Note that in addition to $V_N$, $K(N)$ is also normalized by operators $V_d$ for Hall divisors $d \| N$ if $N$ is not prime, which are analogues of the Atkin-Lehner involutions.) \\

Some results for paramodular forms of degree two rely on their relationship to orthogonal modular forms so we recall this briefly. (See also the discussion in \cite{GN2}.) The space of real antisymmetric $(4 \times 4)$-matrices admits a nondegenerate quadratic form (the Pfaffian) of signature $(3,3)$ which is invariant under conjugation by $\mathrm{SL}_4(\mathbb{R})$, explicitly $$\mathrm{pf} \begin{psmallmatrix} 0 & a & b & c \\ -a & 0 & d & e \\ -b & -d & 0 & f \\ -c & -e & -f & 0 \end{psmallmatrix} = af - be + cd,$$ and this action by conjugation determines the Klein correspondence $\mathrm{SL}_4(\mathbb{R}) =\mathrm{Spin}(\mathrm{pf}).$ If we fix the matrix $J = \begin{psmallmatrix} 0 & 0 & 1 & 0 \\ 0 & 0 & 0 & 1 \\ -1 & 0 & 0 & 0 \\ 0 & -1 & 0 & 0 \end{psmallmatrix}$ with $\mathrm{pf}(J) < 0$ then the symplectic group $\mathrm{Sp}_4(\mathbb{R})$ consists exactly of those matrices which preserve the orthogonal complement $J^{\perp}$ under conjugation, and the Klein correspondence identifies $\mathrm{Sp}_4(\mathbb{R})$ with $\mathrm{Spin}(\mathrm{pf}|_{J^{\perp}}) = \mathrm{Spin}_{3,2}(\mathbb{R})$. Under this identification $K(N)^+$ embeds into the spin group of the even lattice $(\Lambda,N \cdot \mathrm{pf})$ where $\Lambda = \{M \in J^{\perp}: \, \sigma_N M \sigma_N \in \mathbb{Z}^{4 \times 4}\},$ $\sigma_N = \mathrm{diag}(1,1,1,N),$ which is isomorphic to $A_1(N) \oplus II_{2,2}$.  This allows orthogonal modular forms for $(\Lambda, N \cdot \mathrm{pf})$ to be interpreted as paramodular forms. \\

The orthogonal modular variety associated to any even lattice $(\Lambda,Q)$ of signature $(\ell,2)$ admits a natural construction of Heegner divisors. Through the Klein corresponence one obtains from these the \textbf{Humbert surfaces} on $K(N)^+ \backslash \mathbb{H}_2$. (See section 1.3 of \cite{GN2} for background, or the lecture notes \cite{Dolg}.) We use the convention that $\mathcal{H}_D$ is the union of rational quadratic divisors associated to \emph{primitive} lattice vectors of discriminant $D$. (Thus the Heegner divisors of e.g. \cite{B} correspond to $\cup_{r^2 | D} \mathcal{H}_{D/r^2}$.) The surface $\mathcal{H}_D$ has an irreducible component $\mathcal{H}_{D,b}$ for each pair $(\pm b)$ mod $2N$ for which $b^2 \equiv D \, (4N)$. If $a = \frac{b^2 - D}{4N}$ then $\mathcal{H}_{D,b}$ is represented by the surface $$\{\begin{psmallmatrix} \tau & z \\ z & w \end{psmallmatrix} \in \mathbb{H}_2: \; a\tau + bz + Nw = 0\}.$$ In particular if $N$ is a prime then $\mathcal{H}_D$ is either empty or irreducible. When $D=1$ it is useful to know that $\mathcal{H}_1$ is represented by the diagonal. By abuse of notation we also use $\mathcal{H}_D$ to mean the preimages in $\mathbb{H}_2$ and in $K(N) \backslash \mathbb{H}_2$. \\

Two important constructions of paramodular forms arise through the relationship to the orthogonal group and both are described in detail in \cite{GN2}. The first is the \textbf{Gritsenko lift}. Let $J_{k,N}$ denote the space of Jacobi forms of weight $k$ and index $N$. Recall that these are holomorphic functions $\phi : \mathbb{H} \times \mathbb{C} \rightarrow \mathbb{C}$ satisfying the transformations $$\phi\Big( \frac{a\tau + b}{c \tau + d}, \frac{z}{c\tau + d} \Big) = (c \tau + d)^k \mathbf{e}\Big(\frac{Ncz^2}{c \tau + d}\Big) \phi(\tau,z), \; M = \begin{psmallmatrix} a & b \\ c & d \end{psmallmatrix} \in \mathrm{SL}_2(\mathbb{Z})$$ and $$\phi(\tau,z+\lambda \tau + \mu) = \mathbf{e}(-N\lambda^2 \tau - N\lambda z) \phi(\tau,z), \; \; \lambda,\mu \in \mathbb{Z},$$ where we abbreviate $\mathbf{e}(x) = e^{2\pi i x}$, and in which the Fourier series $\phi(\tau,z) = \sum_{n,r \in \mathbb{Z}} \alpha(n,r) q^n \zeta^r$, $q = \mathbf{e}(\tau)$, $\zeta = \mathbf{e}(z)$ may have nonzero coefficients $\alpha(n,r)$ only when $r^2 / N \le 4n$. Additionally, $\phi$ is a cusp form if $\alpha(n,r) = 0$ when $r^2 / N = 4n$. \\

The Gritsenko lift is an additive map $\mathcal{G} : J_{*,N} \rightarrow M_*(K(N))$ which can be defined naturally either by means of Jacobi-Hecke operators or by the theta correspondence. Here we only recall how to compute it. Let $\phi(\tau,z) = \sum_{n,r} \alpha(n,r)q^n \zeta^r \in J_{k,N}$ be a Jacobi form of weight $k \ge 4$. Its Gritsenko lift is $$\mathcal{G}_{\phi}(\begin{psmallmatrix} \tau & z \\ z & w \end{psmallmatrix}) = -\frac{B_k}{2k} \alpha(0,0) (E_k(\tau) + E_k(Nw) - 1) + \sum_{\substack{a,c \ge 1 \\ b^2 \le 4Nac}} \sum_{d | \mathrm{gcd}(a,b,c)} d^{k-1} \alpha\left(\frac{ac}{d^2},\frac{b}{d}\right) q^a r^b s^{Nc},$$ where $E_k(\tau) =  1 - \frac{2k}{B_k} \sum_{n=1}^{\infty} \sigma_{k-1}(n) q^n$ is the scalar Eisenstein series of level $1$ if $k$ is even. (If $k$ is odd then $\alpha(0,0) = 0$ automatically and there is no need to define $E_k$.)

\begin{ex} The paramodular Eisenstein series $\mathcal{E}_k$ is a special case of the Gritsenko lift. Let $k$ be even and let $\phi = E_{k,N}$ be the Jacobi Eisenstein series of index $N$ (as in \cite{EZ}); then $\mathcal{E}_k = -\frac{2k}{B_k}\mathcal{G}_{\phi}$. In particular $\mathcal{E}_k$ is normalized such that its Fourier series has constant term $1$.
\end{ex}

The second construction of paramodular forms we will need is the \textbf{Borcherds lift} of \cite{B}. This is a multiplcative map which sends \emph{nearly-holomorphic} Jacobi forms (where a finite principal part is allowed) to extended paramodular forms with a character. The details appear in chapter 2 of \cite{GN2}. We mention here only that the divisor of a Borcherds lift is a linear combination of Humbert surfaces, and that the divisor, weight and character can be easily read off the principal part of the input Jacobi form. \\

The \textbf{Witt operator} is a restriction to the diagonal: $$P_1 : M_*(K(N)) \longrightarrow M_*(\mathrm{SL}_2(\mathbb{Z}) \times \mathrm{SL}_2(\mathbb{Z})), \;\; P_1 F(\tau_1,\tau_2) = F( \begin{psmallmatrix} \tau_1 & 0 \\ 0 & \tau_2 / N \end{psmallmatrix}).$$ That $P_1$ is well-defined is due to the embedding of groups $$\Phi_1 : \mathrm{SL}_2(\mathbb{Z}) \times \mathrm{SL}_2(\mathbb{Z}) \longrightarrow K(N), \; \; \left[ \begin{psmallmatrix} a_1 & b_1 \\ c_1 & d_1 \end{psmallmatrix}, \begin{psmallmatrix} a_2 & b_2 \\ c_2 & d_2 \end{psmallmatrix} \right] \mapsto \begin{psmallmatrix} a_1 & 0 & b_1 & 0 \\ 0 & a_2 & 0 & b_2 / N \\ c_1 & 0 & d_1 & 0 \\ 0 & Nc_2 & 0 & d_2 \end{psmallmatrix}.$$ One can check directly that $P_1 F(M \cdot (\tau_1,\tau_2)) = F(\Phi_1(M) \cdot \begin{psmallmatrix} \tau_1 & 0 \\ 0 & \tau_2 / N \end{psmallmatrix}) = (c_1 \tau_1 + d_1)^k (c_2 \tau_2 + d_2)^k P_1 F(\tau_1,\tau_2)$ for all $F \in M_k(K(N))$ and $M = \left[ \begin{psmallmatrix} a_1 & b_1 \\ c_1 & d_1 \end{psmallmatrix}, \begin{psmallmatrix} a_2 & b_2 \\ c_2 & d_2 \end{psmallmatrix} \right]$. \\

Satz 4 of \cite{FS} gives for any discriminant $D = a^2 + 4Nb^2$, $a,b \in \mathbb{Z}$ a similar embedding which allows one to restrict paramodular forms of level $N$ to (possibly degenerate, if $D$ is a square) Hilbert modular forms associated to the discriminant $D$. (By ``degenerate" Hilbert modular forms we mean modular forms for subgroups of $\mathrm{SL}_2(\mathbb{Z}) \times \mathrm{SL}_2(\mathbb{Z})$.) For now we focus on the following special case. Suppose $N$ is odd and let $G$ denote the subgroup $$G = \{(M_1,M_2) \in \mathrm{SL}_2(\mathbb{Z}) \times \mathrm{SL}_2(\mathbb{Z}): \; M_1 \equiv M_2 \, (\text{mod} \, 2)\}.$$ This embeds into $K(N)$ by the group homomorphism $$\Phi_4 : G \longrightarrow K(N), \; \; \left[ \begin{psmallmatrix} a_1 & b_1 \\ c_1 & d_1 \end{psmallmatrix}, \begin{psmallmatrix} a_2 & b_2 \\ c_2 & d_2 \end{psmallmatrix} \right] \mapsto \begin{psmallmatrix} a_1 & 0 & 2b_1 & b_1 \\ (a_1 - a_2)/2 & a_2 & b_1 & (b_2 + Nb_1) / 2N \\ (c_1 + Nc_2)/2 & -Nc_2 & d_1 & (d_1 - d_2)/2 \\ -Nc_2 & 2Nc_2 & 0 & d_2 \end{psmallmatrix}.$$ The embedding satisfies $$\Phi_4(M_1,M_2) \cdot \begin{psmallmatrix} 2\tau_1 & \tau_1 \\ \tau_1 & \tau_1/2 + \tau_2 / 2N \end{psmallmatrix} = \begin{psmallmatrix} 2 (M_1 \cdot \tau_1) & M_1 \cdot \tau_1 \\ M_1 \cdot \tau_1 & \frac{1}{2}(M_1 \cdot \tau_1) + \frac{1}{2N} (M_2 \cdot \tau_2) \end{psmallmatrix}$$ so we get an associated pullback map: $$P_4 : M_*(K(N)) \longrightarrow M_*(G), \; \; P_4 F(\tau_1,\tau_2) = F( \begin{psmallmatrix} 2\tau_1 & \tau_1 \\ \tau_1 & \tau_1/2 + \tau_2 / 2N \end{psmallmatrix}).$$ 

These definitions are natural in the interpretation of orthogonal groups (here, $G$ is essentially the orthogonal group of the lattice $A_1 \oplus A_1(-1) \oplus II_{1,1}$). Note that if $N$ is prime then $\mathcal{H}_4$ is the orbit of $$\{\begin{psmallmatrix} 2\tau_1 & \tau_1 \\ \tau_1 & \tau_1/2 + \tau_2/2N \end{psmallmatrix}: \, \tau_1,\tau_2 \in \mathbb{H}\} = \{\begin{psmallmatrix} \tau & z \\ z & w \end{psmallmatrix} \in \mathbb{H}_2: \; \tau = 2z\}$$ under $K(N)^+$. Thus a symmetric or antisymmetric paramodular form (e.g. a Borcherds product or a Gritsenko lift) $F$ for which $P_4 F = 0$ vanishes everywhere on $\mathcal{H}_4$. \\ 

The behavior of the map $P_4$ under the involution $V_N$ is easy to describe:

\begin{lem} Let $F$ be a paramodular form of odd level $N$. Then $P_4 (F | V_N)(\tau_1,\tau_2) = P_4F(\tau_2,\tau_1).$
\end{lem}
\begin{proof} Fix the matrix $u = \begin{psmallmatrix} 2 & N \\ 1 & \frac{N+1}{2} \end{psmallmatrix} \in \mathrm{SL}_2(\mathbb{Z})$. Since the upper-right entry in $u$ is a multiple of $N$, it follows that the conjugation map $U = \begin{psmallmatrix} u & 0 \\ 0 & (u^{-1})^T \end{psmallmatrix}$ lies in $K(N)$. We find $$U \cdot V_N \cdot \begin{psmallmatrix} 2\tau_1 & \tau_1 \\ \tau_1 & \tau_1/2 + \tau_2 / 2N \end{psmallmatrix} = \begin{psmallmatrix} 2\tau_2 & \tau_2 \\ \tau_2 & \tau_2/2 + \tau_1 / 2N \end{psmallmatrix}$$ and the claim follows because $F$ is invariant under $U$.
\end{proof}

\section{A ring of degenerate Hilbert modular forms}

In this section we give a more careful study of modular forms for the group $G$ considered earlier. The structure of $M_*(G)$ is surely well-known (and for example the underlying surface was considered in detail in \cite{K}, section 3) but because of the frequent need to refer to it we give a complete account. Note that a related problem was solved in \cite{MM} for the group of pairs $(M_1,M_2)$ with $(M_1^{-1})^T \equiv M_2$ mod $3$ by means of invariant theory (Molien series). Their approach would also apply here but the structure of $M_*(G)$ is much simpler. \\

Bear in mind that (unlike the case of true Hilbert modular forms) Koecher's principle does not apply to the action of $G$ on $\mathbb{H} \times \mathbb{H}$ because the Satake boundary has components of codimension one. To account for this we define Phi operators $\Phi_1,\Phi_2$ by $$\Phi_1 f(\tau) = \lim_{z \rightarrow i \infty} f(z,\tau), \; \; \Phi_2 f(\tau) = \lim_{z \rightarrow i \infty} f(\tau,z).$$ A holomorphic function $f$ on $\mathbb{H} \times \mathbb{H}$ satisfying $$f(M_1 \cdot \tau_1, M_2 \cdot \tau_2) = (c_1 \tau_1 + d_1)^k (c_2 \tau_2 + d_2)^k f(\tau_1,\tau_2) \; \text{for all} \; \Big(\begin{psmallmatrix} a_1 & b_1 \\ c_1 & d_1 \end{psmallmatrix}, \begin{psmallmatrix} a_2 & b_2 \\ c_2 & d_2 \end{psmallmatrix} \Big) \in G$$ which is also holomorphic at the cusps, i.e. for which $\Phi_1 f$ and $\Phi_2 f$ are both holomorphic modular forms (of level $\Gamma_0(2)$), is a \textbf{modular form} for $G$. $f$ is a \textbf{cusp form} if $\Phi_1 f$, $\Phi_2 f$ are both zero. \\

Let $E_2(\tau) = 1 - 24 \sum_{n=1}^{\infty} \sigma_1(n) q^n$, $q = e^{2\pi i \tau}$ denote the (non-modular) Eisenstein series of weight two, and define $$e_1(\tau) = 2 E_2(2\tau) - E_2(\tau), \; e_2(\tau) = E_2(\tau) - \frac{1}{2}E_2(\tau/2).$$ Then $e_1,e_2$ are algebraically independent modular forms of level $\Gamma(2)$ and they generate the graded ring $M_*(\Gamma(2))$. Their behavior under the full modular group is $$e_1 | T = e_1, \; e_2 | T = e_1 - e_2, \; e_1 | S = -e_2, \; e_2|S = -e_1,$$ where $S = \begin{psmallmatrix} 0 & -1 \\ 1 & 0 \end{psmallmatrix}$, $T = \begin{psmallmatrix} 1 & 1 \\ 0 & 1 \end{psmallmatrix}$ and where $|$ is the Petersson slash operator. Define the products $$f_{ij}(\tau_1,\tau_2) = e_i(\tau_1)e_j(\tau_2), \; \; i,j \in \{1,2\}.$$

A modular form $f(\tau_1,\tau_2)$ is \textbf{symmetric} if $f(\tau_1,\tau_2) = f(\tau_2,\tau_1)$ and \textbf{antisymmetric} if $f(\tau_1,\tau_2) = -f(\tau_2,\tau_1)$.

\begin{lem}\label{H4} The ring of symmetric modular forms for $G$ is a polynomial ring in three variables: $$M_*^{sym}(G) = \mathbb{C}[X_2,X_4,\Delta_6],$$ where $X_2 = \frac{4}{3}(f_{11} + f_{22}) - \frac{2}{3}(f_{12} + f_{21})$ has weight two, $X_4 = \frac{1}{144} (f_{12} - f_{21})^2$ has weight four, and where \begin{align*} \Delta_6(\tau_1,\tau_2) &= \eta(\tau_1)^{12} \eta(\tau_2)^{12} \\ &= \frac{1}{2916} (f_{11} + f_{12} + f_{21} + f_{22}) (f_{11} - 2f_{12} - 2f_{21} + 4f_{22}) (4 f_{11} - 2f_{12} - 2f_{21} + f_{22}) \end{align*} is a cusp form of weight six. Here $\eta(\tau) = q^{1/24} \prod_{n=1}^{\infty} (1 - q^n)$ is the Dedekind eta function.
\end{lem}
\begin{proof} It is clear that $X_2,X_4,\Delta_6$ transform correctly under the diagonal action of $\mathrm{SL}_2(\mathbb{Z})$ and therefore define modular forms for $G$. The claimed expression for $\Delta_6$ in terms of the $f_{ij}$ follows from factoring $$\eta^{12} = \frac{1}{54} (e_1 + e_2)(2e_1 - e_2)(2e_2 - e_1).$$ The multiples are chosen to make the Fourier coefficients coprime integers. \\

We will prove that every symmetric modular form $f$ is a polynomial in $X_2,X_4,\Delta_6$ by induction on its weight $k$. If $k \le 0$ then $f$ is constant. Otherwise, applying either Phi operator $\Phi \in \{\Phi_1,\Phi_2\}$ yields $$\Phi X_2 = e_1, \; \Phi X_4 = \frac{1}{144}(e_1/2 - e_2)^2, \; \Phi \Delta_6 = 0.$$ The forms $e_1, (e_1/2 - e_2)^2$ generate the ring $M_*(\Gamma_0(2))$ (as one can show by computing dimensions) so there is some polynomial $P$ for which $\Phi[f - P(X_2,X_4)] = 0$ for both Phi operators. Since $\Delta_6$ has no zeros on $\mathbb{H} \times \mathbb{H}$ and zeros of minimal order along both boundary components, it follows that $\frac{f - P(X_2,X_4)}{\Delta_6}$ is a holomorphic modular form of weight $k-6$. By induction, $\frac{f - P(X_2,X_4)}{\Delta_6}$ and therefore also $f$ is a polynomial in $X_2,X_4,\Delta_6$. Since $M_*^{sym}(G)$ has Krull dimension $3$, the forms $X_2,X_4,\Delta_6$ must be algebraically independent. (One can also prove this by considering the values along the diagonal $(\tau,\tau)$, since $X_2(\tau,\tau) = E_4(\tau)$ and $\Delta_6(\tau,\tau) = \Delta(\tau)$ are algebraically independent and $X_4(\tau,\tau) \equiv 0$.)
\end{proof}

\begin{lem}\label{H4anti} The $M_*^{sym}(G)$-module of antisymmetric modular forms is free with a single generator $$X_8 = \frac{1}{81} (f_{11} - f_{22}) (f_{12} - f_{21}) (f_{12} + f_{21} - f_{11}) (f_{12} + f_{21} - f_{22})$$ in weight $8$.
\end{lem}
In other words we have the Hironaka decomposition $$M_*(G) = \mathbb{C}[X_2,X_4,\Delta_6] \oplus X_8 \mathbb{C}[X_2,X_4,\Delta_6].$$ By computing Fourier expansions explicitly one can show that the quadratic equation $X_8$ satisfies over $M_*^{sym}(G)$ is $$R : X_8^2 = X_2^4 X_4^2 - 128 X_2^2 X_4^3 + 4096 X_4^4 + 4 X_2^3 X_4 \Delta_6 - 2304 X_2 X_4^2 \Delta_6 - 6912 X_4 \Delta_6^2,$$ such that the full ring of modular forms is $$M_*(G) = \mathbb{C}[X_2,X_4,\Delta_6,X_8] / R.$$
\begin{proof} Suppose $f$ is an antisymmetric form. Letting $\tau$ tend to $i \infty$ in the equation $$\Phi_1 f(\tau) = \lim_{z \rightarrow i \infty} f(z,\tau) = -\lim_{z \rightarrow i \infty} f(\tau,z) = -\Phi_2 f(\tau)$$ and its images under the diagonal action of $\mathrm{SL}_2(\mathbb{Z})$ shows that $\Phi_1 f, \Phi_2 f$ are cusp forms of level $\Gamma_0(2)$. The ideal of cusp forms in $M_*(\Gamma_0(2))$ is principal, generated in weight $8$ by $\eta(\tau)^8 \eta(2\tau)^8$. \\

It is straightforward to check that $X_8$ is a modular form for $G$ (i.e. it transforms correctly under the diagonal action of $\mathrm{SL}_2(\mathbb{Z})$), that it is antisymmetric due to the factor $f_{12} - f_{21}$ in its product expression, and that it has image $\Phi_1 X_8(\tau) = \eta(\tau)^8 \eta(2 \tau)^8$ under the Phi operator. By the previous paragraph, there exists a polynomial $P$ such that $\Phi_1 [f - X_8 P(X_2,X_4)] = 0$, and by antisymmetry $\Phi_2 [f - X_8 P(X_2,X_4)] = 0$. The quotient $\frac{f - X_8 P(X_2,X_4)}{\Delta_6}$ is then a holomorphic, antisymmetric modular form of smaller weight than $f$, so by an induction argument (similar to the previous lemma) we find $\frac{f - X_8 P(X_2,X_4)}{\Delta_6} \in X_8 \cdot \mathbb{C}[X_2,X_4,\Delta_6]$ and therefore $f \in X_8 \mathbb{C}[X_2,X_4,\Delta_6]$.
\end{proof}

The pullback $P_4$ is never surjective. The fact that the surface $\overline{\mathcal{H}_4}$ inherits one-dimensional boundary components from the Siegel threefold restricts the modular forms that can arise as pullbacks of paramodular forms. What this means explicitly is the following:

\begin{prop} Suppose $f = P_4 F$ for a paramodular form $F$ of any (odd) level $N$. Then $\Phi_1 f(\tau/2)$ and $\Phi_2 f(\tau/2)$ are modular forms of level $1$.
\end{prop}
The analogous statement for meromorphic paramodular forms is not true.
\begin{proof} Write out the Fourier expansion of $F$: $$F(\begin{psmallmatrix} \tau & z \\ z & w \end{psmallmatrix}) = \sum_{a,b,c} \alpha(a,b,c) q^a r^b s^{Nc}, \; q = e^{2\pi i \tau}, \; r = e^{2\pi i z}, \; s = e^{2\pi i w}.$$ Then \begin{align*} \Phi_2 P_4 F(\tau/2) = \lim_{z \rightarrow i \infty} P_4 F(\tau/2,z) &= \lim_{z \rightarrow i \infty} F(\begin{psmallmatrix} \tau & \tau/2 \\ \tau/2 & \tau/4 + z/2N \end{psmallmatrix}) \\ &= \sum_{a,b=0}^{\infty} \alpha(a,b,0) q^{a + b/2}. \end{align*} By Koecher's principle, $\alpha(a,b,0) = 0$ unless $b = 0$ so $$\Phi_2 P_4 F(\tau/2) = \sum_{a=0}^{\infty} \alpha(a,0,0)q^a = \lim_{z \rightarrow i \infty} F( \begin{psmallmatrix} \tau & 0 \\ 0 & z \end{psmallmatrix}) = \Phi F(\tau)$$ is the image of $F$ under the usual Siegel Phi operator. This can be shown to be a modular form for the full group $\mathrm{SL}_2(\mathbb{Z})$ using the diagonal embedding of $\mathrm{SL}_2(\mathbb{Z})^2$ in $K(N)$. The proof for $\Phi_1 f(\tau/2)$ is similar. (One can also decompose $F$ into symmetric and antisymmetric parts under the Fricke involution $V_N$ to reduce to the cases $\Phi_1 f = \pm \Phi_2 f$.)
\end{proof}

It is clear that the forms $f \in M_*(G)$ for which $\Phi_1 f(\tau/2)$ and $\Phi_2 f(\tau/2)$ have level one form a graded subring of $M_*(G)$. We denote it by $A_*$ and we let $A_*^{sym}$ denote the subring of symmetric forms in $A_*$.

\begin{prop}\label{genA} (i) $A_*^{sym}$ is generated by five forms $$X_2^2 - 48X_4, \, X_2^3 - 72X_2 X_4, \, \Delta_6, \, X_2 \Delta_6, \, X_4 \Delta_6$$ in weights $4,6,6,8,10$. \\ (ii) $A_*$ is generated by the five generators of $A_*^{sym}$ together with the antisymmetric forms $X_4 X_8$, $\Delta_6 X_8$ and $X_2 \Delta_6 X_8$.
\end{prop}
\begin{proof} (i) It was shown in the proof of Lemma~\ref{H4} that the Phi operator $$\Phi_2 : M_*^{sym}(G) \rightarrow M_*(\Gamma_0(2))$$ is surjective. Therefore we compute $$\mathrm{dim}\, M_k^{sym}(G) - \mathrm{dim}\, A_k^{sym} = \mathrm{dim} \, M_k(\Gamma_0(2)) - \mathrm{dim}\, M_k(\mathrm{SL}_2(\mathbb{Z})) = \lceil k/6 \rceil$$ for all even $k$. Using $\sum_{k=0}^{\infty} \lceil 2k/6 \rceil t^{2k} = \frac{t^2}{(1 - t^2)(1 - t^6)}$ and the structure $M_*^{sym}(G) = \mathbb{C}[X_2,X_4,\Delta_6]$ we find the Hilbert series \begin{align*} \mathrm{Hilb} \, A_*^{sym} &= \sum_{k=0}^{\infty} \mathrm{dim} \, A_k^{sym} t^k \\ &= \frac{1}{(1 - t^2)(1 - t^4)(1 - t^6)} - \frac{t^2}{(1 - t^2)(1 - t^6)} \\ &= \frac{1 + t^8 + t^{10}}{(1 - t^4)(1 - t^6)^2}. \end{align*} Applying either Phi operator to $X_2^2 - 48X_4$ and $X_2^3 - 72X_2X_4$ yields $E_4(2\tau)$ and $E_6(2\tau)$, respectively, so these are contained in $A_*^{sym}$; and this condition is trivial for $\Delta_6$, $X_2 \Delta_6$ and $X_4 \Delta_6$ which are cusp forms. Therefore to prove (i) one only needs to see that the subalgebra generated by these five forms has the Hilbert series $\frac{1 + t^8 + t^{10}}{(1 - t^4)(1 - t^6)^2}$. This can be checked in Macaulay2 or by hand. \\

(ii) Similarly the Phi operator $\Phi_2 : M_*^{anti}(G) = X_8 \cdot M_*^{sym}(G) \rightarrow S_*(\Gamma_0(2))$ is surjective onto cusp forms (as shown in the proof of Lemma~\ref{H4anti}) so we find $$\mathrm{dim}\, M_k^{anti}(G) - \mathrm{dim}\, A_k^{anti} = \mathrm{dim}\, S_k(\Gamma_0(2)) - \mathrm{dim}\, S_k(\mathrm{SL}_2(\mathbb{Z})) = \lceil k/6 \rceil - 1.$$ Thus the graded module $A_*^{anti}$ has Hilbert series $$\sum_{k=0}^{\infty} \mathrm{dim}\, A_k^{anti} t^k = \frac{t^8}{(1 - t^2)(1 - t^4)(1 - t^6)} - \frac{t^8}{(1 - t^2)(1 - t^6)} = \frac{t^{12}}{(1 -t ^2)(1 - t^4)(1 - t^6)}$$ and therefore $$\mathrm{Hilb} \, A_* = \frac{1 + t^8 + t^{10}}{(1 - t^4)(1-t^6)^2} + \frac{t^{12}}{(1 - t^2)(1-t^4)(1-t^6)} = \frac{1 - t^2 + t^6 + t^{12}}{(1 - t^2)(1 - t^4)(1 - t^6)}.$$ Note that $\Phi_1(X_4 X_8) = \Delta(2\tau)$ i.e. $X_4 X_8 \in A_{12}$ and that $\Delta_6 X_8$ and $X_2 \Delta_6 X_8$ are cusp forms, so all three are indeed contained in $A_*$. We checked again that the subring of $M_*(G)$ generated by the eight forms in the claim has the same Hilbert series as $A_*$ so these forms generate $A_*$.
\end{proof}

\begin{rem} One can ask for what levels $N$ the ring $A_*$ is exactly the image of $P_4$. It seems reasonably likely that for any odd $N \ge 3$ one has $A_*^{sym} = \mathrm{im}(P_4)$. The five generators of $A_*^{sym}$ given above can be shown to be theta lifts (essentially the Doi-Naganuma lift for degenerate Hilbert modular forms) so one might try to realize them as pullbacks of families of Gritsenko lifts. We only consider levels $N=5,7$ below so we leave this question aside. As for antisymmetric forms, for $N \ge 3$ one can always realize the non-cusp form $X_4 X_8$ as a pullback of Gritsenko and Nikulin's pullbacks (\cite{GN2}, Remark 4.4) of the Borcherds form $\Phi_{12}$ which is automorphic under the orthogonal group $O(26,2)$. To see this it is enough to observe that Gritsenko and Nikulin's pullbacks are not cusp forms so they are not allowed to vanish along $\mathcal{H}_4$ (or any other Humbert surface of square discriminant for that matter). Whether the cusp forms $\Delta_6 X_8$ and $X_2 \Delta_6 X_8$ arise as pullbacks seems to be less clear.
\end{rem}

\section{Paramodular forms of level $5$}

In this section we compute the graded ring of paramodular forms of level $5$ in terms of the Eisenstein series $\mathcal{E}_4$ and $\mathcal{E}_6$ and the Borcherds products and Gritsenko lifts listed in tables 1 and 2 below.

\begin{longtable}{|l|l|l|l|l|}
\caption{Some Borcherds products of level 5} \\
\hline
   Name  &    Weight     &  Symmetric &  Divisor                                                       \\ \hline
       $b_5$ & $5$  & no & $7\mathcal{H}_1 + \mathcal{H}_4$ \\ \hline
       $b_8$ & $8$  & yes & $4\mathcal{H}_1 + 2\mathcal{H}_4 + 2\mathcal{H}_5$ \\ \hline
       $b_{12}$ & $12$  & no & $3\mathcal{H}_5 + \mathcal{H}_{20}$ \\ \hline
       $b_{14}$ & $14$ & no & $10\mathcal{H}_1 + \mathcal{H}_5 + \mathcal{H}_{20}$ \\ \hline
\end{longtable}

Tables of Borcherds products (including those with character) are available in chapter 7 of \cite{M} or appendix A of \cite{W}. In the notation of \cite{W} these are the forms $b_5 = \psi_5$, $b_8 = \psi_4^2$, $b_{12} = \psi_{12}$ and $b_{14} = \frac{b_5^2 b_{12}}{b_8} = \psi_4^{-2} \psi_5^2 \psi_{12}$. For a more modern approach to computing paramodular Borcherds products see also \cite{PSY}. (Bear in mind that \cite{PSY} focuses on cusp forms, and that $b_{12}$ is not a cusp form since its divisor does not contain any Humbert surface of square discriminant.)

\begin{longtable}{|l|l|l|}
\caption{Some Gritsenko lifts of level 5} \\
\hline
   Name  &    Weight &    Input Jacobi form                                                      \\ \hline
       $g_6$ & $6$ & $\begin{aligned} &\quad \Big( (\zeta^{-4} + \zeta^4) - 2(\zeta^{-3} + \zeta^3) - 8(\zeta^{-2} + \zeta^2) + 34(\zeta^{-1} + \zeta) - 50 \Big) q \\ &\quad\quad + \Big( (\zeta^{-6} + \zeta^6) - 20(\zeta^{-5} + \zeta^5) + 62(\zeta^{-4} + \zeta^4) - 52(\zeta^{-3} + \zeta^3) \\ &\quad\quad - 33(\zeta^{-2} + \zeta^2)+ 72(\zeta^{-1} + \zeta) - 60 \Big) q^2 + O(q^3) \end{aligned}$ \\ \hline
       $g_7$ & $7$ & $\begin{aligned} &\quad \Big( (\zeta^{-4} - \zeta^4) + 6(\zeta^{-3} - \zeta^3) - 34(\zeta^{-2} - \zeta^2) + 46(\zeta^{-1} - \zeta) \Big) q \\ &\quad\quad + \Big( -1(\zeta^{-6} - \zeta^6)  - 108(\zeta^{-4} - \zeta^4) + 416(\zeta^{-3} - \zeta^3) \\ &\quad\quad - 549(\zeta^{-2} - \zeta^2) + 288(\zeta^{-1} - \zeta) \Big)q^2 + O(q^3) \end{aligned}$ \\ \hline
       $g_8$ & $8$ & $\begin{aligned} &\quad \Big( (\zeta^{-4} + \zeta^4) + 12(\zeta^{-2} + \zeta^2) - 64(\zeta^{-1} + \zeta) + 102 \Big)q \\ &\quad\quad + \Big( (\zeta^{-6} + \zeta^6) + 24(\zeta^{-5} + \zeta^5) - 222(\zeta^{-4} + \zeta^4) + 312(\zeta^{-3} + \zeta^3) \\ &\quad\quad + 495(\zeta^{-2} + \zeta^2) - 1872(\zeta^{-1} + \zeta) + 2524 \Big) q^2 + O(q^3) \end{aligned}$ \\ \hline
       $g_{10}$ & $10$ & $\begin{aligned} &\quad \Big(  (\zeta^{-3} + \zeta^3) - 6(\zeta^{-2} + \zeta^2) + 15(\zeta^{-1} + \zeta) - 20 \Big)q \\ &\quad\quad + \Big( 14(\zeta^{-5} + \zeta^5) - 26(\zeta^{-4} + \zeta^4) - 66(\zeta^{-3} + \zeta^3) \\ &\quad\quad + 216(\zeta^{-2} + \zeta^2) - 204(\zeta^{-1} + \zeta) + 132 \Big)q^2 + O(q^3) \end{aligned}$ \\ \hline
\end{longtable}

There are several ways to compute Jacobi forms. We used the algorithm of \cite{W} to compute the equivalent vector-valued modular forms. The motivation for these choices is that $g_6,g_8,g_{10}$ pullback to the forms $\Delta_6$, $X_2 \Delta_6$ and $X_4 \Delta_6$ under $P_4$ and that $g_7 / b_5$ is (as we will show) meromorphic, has no poles except a double pole on $\mathcal{H}_1$, and is nonvanishing along the surfaces $\mathcal{H}_4,\mathcal{H}_5,\mathcal{H}_{20}$. In particular, given a paramodular form $F$ which vanishes to a particular order along the diagonal, we can produce a chain of holomorphic forms $F, F (g_7/b_5), F(g_7/b_5)^2,...$ while keeping control of the orders along these Humbert surfaces. \\

We should remark that Borcherds' additive singular theta lift (Theorem 14.3 of \cite{B}) implies that, for even $k \ge 4$, one can construct a meromorphic paramodular form of weight $k$ with only poles of order $k$ along an arbitrary collection of Humbert surfaces $\mathcal{H}_D$. (This will be the theta lift of a linear combination of Poincar\'e series of negative index and weight $k-1/2$.) The argument fails in weight $k=2$ and the fact that we can produce such a paramodular form for the diagonal (i.e. $g_7/b_5$) by other means turns out to be rather useful. \\

Any paramodular form $F( \begin{psmallmatrix} \tau & z \\ z & w \end{psmallmatrix})$ which is not identically zero can be expanded as $$F( \begin{psmallmatrix} \tau & z \\ z & w/N \end{psmallmatrix}) = g(\tau,w) z^n + O(z^{n+1})$$ where $g \ne 0$. It is clear that $g$ is invariant under $\tau \mapsto \tau+1$ and $w \mapsto w+1$. Considering the behavior of $F$ under $\begin{psmallmatrix} 0 & 0 & -1 & 0 \\ 0 & 1 & 0 & 0 \\ 1 & 0 & 0 & 0 \\ 0 & 0 & 0 & 1 \end{psmallmatrix}, \begin{psmallmatrix} 1 & 0 & 0 & 0 \\ 0 &0 & 0 & -1/N \\ 0 & 0 & 1 & 0 \\ 0 & 1 & 0 & 0 \end{psmallmatrix}$ shows that $g$ satisfies $$g(-1/\tau,w) = \tau^{k+n} g(\tau,w), \; \; g(\tau,-1/w) = w^{k+n} g(\tau,w),$$ and boundedness of $g$ as either $\tau$ or $w$ tends to $i \infty$ follows from that of $F$. (In fact if $n > 0$ then $g$ tends to zero at the cusps.) Therefore $g$ is a modular form for $\mathrm{SL}_2(\mathbb{Z}) \times \mathrm{SL}_2(\mathbb{Z})$ of weight $k+n$ and a cusp form if $n > 0$. \\

The construction of $g$ above is generalized to arbitrary Humbert surfaces by the \textbf{quasi-pullback}. If $F$ is a paramodular form of weight $k$ with a zero of order $m$ on the Humbert surface $\mathcal{H}_D$ then one obtains a Hilbert modular form for $\mathbb{Q}(\sqrt{D})$ of weight $k+m$, which is a cusp form if $m > 0$. We will never actually need to compute quasi-pullbacks (although we will often use the fact that they exist) so we omit the definition which is more natural in the interpretation as orthogonal modular forms. See section 8 of \cite{GHS}, especially Theorem 8.11. \\

Besides the pullbacks $P_1$ and $P_4$ to the diagonal and to $\mathcal{H}_4$ we will also need to evaluate paramodular forms along the Humbert surface $\mathcal{H}_5$. To be completely explicit we fix the pullback operator $P_5$ below.

\begin{lem}\label{P5} Let $\lambda = \frac{5 - \sqrt{5}}{10}$ and $\lambda' = \frac{5 + \sqrt{5}}{10}$. The Humbert surface $\mathcal{H}_5$ is the orbit of $$\left\{ \begin{psmallmatrix} \tau_1 + \tau_2 & \lambda \tau_1 + \lambda' \tau_2 \\ \lambda \tau_1 + \lambda' \tau_2 & \lambda^2 \tau_1 + (\lambda')^2 \tau_2 \end{psmallmatrix}: \; \tau_1, \tau_2 \in \mathbb{H}\right\}$$ under the extended paramodular group $K(5)^+$. For any paramodular form $F$, $$P_5 F(\tau_1,\tau_2) = F\left( \begin{psmallmatrix} \tau_1 + \tau_2 & \lambda \tau_1 + \lambda' \tau_2 \\ \lambda \tau_1 + \lambda' \tau_2 & \lambda^2 \tau_1 + (\lambda')^2 \tau_2 \end{psmallmatrix}\right)$$ is a Hilbert modular form for the full group $\mathrm{SL}_2(\mathcal{O}_K)$ (where $\mathcal{O}_K$ is the ring of integers of $K = \mathbb{Q}(\sqrt{5})$). Moreover the pullback $P_5$ preserves cusp forms and it sends $V_5$-symmetric forms to symmetric Hilbert modular forms.
\end{lem}
\begin{proof} The matrices $\begin{psmallmatrix} \tau & z \\ z & w \end{psmallmatrix} \in \mathbb{H}_2$ which satisfy $\tau - 5z + 5w = 0$ represent $\mathcal{H}_5$ under the action of $K(5)^+$. (See e.g. \cite{M}, Example 3.6.2.) Since our choice of $\lambda$ satisfies $5\lambda^2 - 5 \lambda + 1 = 0$ we see that $\begin{psmallmatrix} \tau_1 + \tau_2 & \lambda \tau_1 + \lambda' \tau_2 \\ \lambda \tau_1 + \lambda' \tau_2 & \lambda^2 \tau_1 + (\lambda')^2 \tau_2 \end{psmallmatrix}$ parameterizes exactly those matrices as $\tau_1,\tau_2$ run through $\mathbb{H}$. \\

Let $F$ be a paramodular form. To show that $P_5 F$ is a Hilbert modular form it is enough to consider the translations $T_b : (\tau_1,\tau_2) \mapsto (\tau_1+b,\tau_2 + b')$, $b \in \mathcal{O}_K$ and the inversion $S : (\tau_1,\tau_2) \mapsto (-1/\tau_1,-1/\tau_2)$, since these generate $\mathrm{SL}_2(\mathcal{O}_K)$. (This is true for all real-quadratic fields $K$ by a theorem of Vaserstein. For $K = \mathbb{Q}(\sqrt{5})$ this is easier to prove since $\mathcal{O}_K$ is euclidean.) The invariance of $P_5 F$ under $T_b$ follows from the invariance of $F$ under the translation by $\begin{psmallmatrix} \mathrm{Tr}(b) & \mathrm{Tr}(\lambda b) \\ \mathrm{Tr}(\lambda b) & \mathrm{Tr}(\lambda^2 b) \end{psmallmatrix} \in \begin{psmallmatrix} \mathbb{Z} & \mathbb{Z} \\ \mathbb{Z} & \frac{1}{5}\mathbb{Z} \end{psmallmatrix}$ , and the transformation of $P_5 F$ under $S$ follows from that of $F$ under $$\begin{psmallmatrix} 0 & 0 & 2 & 1 \\ 0 & 0 & 1 & 3/5 \\ -3 & 5 & 0 & 0 \\ 5 & -10 & 0 & 0 \end{psmallmatrix} \in K(5),$$ which maps $\begin{psmallmatrix} \tau_1 + \tau_2 & \lambda \tau_1 + \lambda' \tau_2 \\ \lambda \tau_1 + \lambda' \tau_2 & \lambda^2 \tau_1 + (\lambda')^2 \tau_2 \end{psmallmatrix}$ to $\begin{psmallmatrix} -\tau_1^{-1} -\tau_2^{-1} & -\lambda \tau_1^{-1} - \lambda' \tau_2^{-1} \\ -\lambda \tau_1^{-1} - \lambda' \tau_2^{-1} & -\lambda^2 \tau_1^{-1} - (\lambda')^2 \tau_2^{-1} \end{psmallmatrix}$.
\end{proof}

One would like to find paramodular forms which are mapped under $P_5$ to generators of the ring $M_*(\mathrm{SL}_2(\mathcal{O}_K))$ of Hilbert modular forms for $K = \mathbb{Q}(\sqrt{5})$. This is impossible (for example the Eisenstein series, a Hilbert modular form of weight two, does not lie in the image); still, to determine what forms do lie in the image of $P_5$ it is helpful to have the structure theorem (due to Gundlach \cite{Gu}): $$M_*(\mathrm{SL}_2(\mathcal{O}_K)) = \mathbb{C}[\mathbf{E}_2,\mathbf{E}_6,s_5^2] \oplus s_5 \mathbb{C}[\mathbf{E}_2,\mathbf{E}_6,s_5^2] \oplus s_{15} \mathbb{C}[\mathbf{E}_2,\mathbf{E}_6,s_5^2] \oplus s_5 s_{15} \mathbb{C}[\mathbf{E}_2,\mathbf{E}_6,s_5^2].$$ Here, $\mathbf{E}_2$ and $\mathbf{E}_6$ are the Eisenstein series and $s_5,s_{15}$ are antisymmetric and symmetric cusp forms of weights $5$ and $15$, respectively, which can be constructed as Borcherds products (as in the example at the end of \cite{BB}; in that notation $s_5 = \Psi_1$ and $s_{15} = \Psi_5$). Thus the direct sum above is the decomposition into symmetric and antisymmetric forms of even and odd weights. \\

Since the reduction argument we will follow uses Borcherds products $b_5,b_8,b_{12}$ whose divisors are supported on the surfaces $\mathcal{H}_1,\mathcal{H}_4,\mathcal{H}_5,\mathcal{H}_{20}$, it is helpful to know that paramodular forms of certain weights have forced zeros there.

\begin{lem}\label{level5orders} (i) Every even-weight paramodular form has even order on $\mathcal{H}_1$ and $\mathcal{H}_4$. \\ (ii) Every odd-weight paramodular form has odd order on $\mathcal{H}_1$ and $\mathcal{H}_4$. \\ (iii) Every $V_5$-symmetric even-weight paramodular form and every $V_5$-antisymmetric odd-weight paramodular form has even order on $\mathcal{H}_5$ and $\mathcal{H}_{20}$. \\ (iv) Every $V_5$-antisymmetric even-weight paramodular form and every $V_5$-symmetric odd-weight paramodular form has odd order on $\mathcal{H}_5$ and $\mathcal{H}_{20}$.
\end{lem}
\begin{proof} Claims (i) and (ii) follow from the fact that the quasi-pullback to $\mathcal{H}_1$ or to $\mathcal{H}_4$ must have even weight, since $\mathrm{SL}_2(\mathbb{Z}) \times \mathrm{SL}_2(\mathbb{Z})$ and $G$ do not admit nonzero modular forms of odd weight. \\

Claim (iii) follows from claim (iv). Indeed if $F$ is a paramodular form as in claim (iii) then $b_{12} F$ is a paramodular form of the type considered in claim (iv), so $$\mathrm{ord}_{\mathcal{H}_5}(b_{12} F) = 3 + \mathrm{ord}_{\mathcal{H}_5} F, \; \; \mathrm{ord}_{\mathcal{H}_{20}}(b_{12} F) = 1 + \mathrm{ord}_{\mathcal{H}_{20}} F$$ are odd. Therefore we only need to show claim (iv). In fact it is enough to prove this when $F$ is antisymmetric and has even weight, since a similar argument with the product $b_5 F$ will then cover the case that $F$ is symmetric and has odd weight. \\

Let $F$ be an antisymmetric even-weight paramodular form. First we show that the pullback $P_5 F$ is zero. Define $u = \begin{psmallmatrix} 2 & -5 \\ -1 & 2 \end{psmallmatrix} \in \mathrm{GL}_2(\mathbb{Z})$. Since the upper-right entry of $u$ is a multiple of $5$, the paramodular group contains the conjugation map $R_u = \begin{psmallmatrix} u & 0 \\ 0 & (u^{-1})^T \end{psmallmatrix} : \tau \mapsto u\tau u^T.$ The composition $R_u V_5$ fixes the Humbert surface $\mathcal{H}_5$ pointwise, i.e. for any $\tau_1,\tau_2 \in \mathbb{H}$, letting $\lambda = \frac{5 - \sqrt{5}}{10}$ as above, one can compute $$R_u V_5 \cdot \begin{psmallmatrix} \tau_1 + \tau_2 & \lambda \tau_1 + \lambda' \tau_2 \\ \lambda \tau_1 + \lambda' \tau_2 & \lambda^2 \tau_1 + (\lambda')^2 \tau_2 \end{psmallmatrix} = R_u \cdot \begin{psmallmatrix} 5\lambda^2 \tau_1 + 5(\lambda')^2 \tau_2 & -\lambda \tau_1 - \lambda' \tau_2 \\ -\lambda \tau_1 - \lambda' \tau_2 & (1/5)\tau_1 + (1/5)\tau_2 \end{psmallmatrix} =  \begin{psmallmatrix} \tau_1 + \tau_2 & \lambda \tau_1 + \lambda' \tau_2 \\ \lambda \tau_1 + \lambda' \tau_2 & \lambda^2 \tau_1 + (\lambda')^2 \tau_2 \end{psmallmatrix}.$$ For antisymmetric forms $F$ of weight $k$ and $\tau \in \mathcal{H}_5$ we find $$F(R_u V_5 \cdot \tau) = (-1)^k (F | R_u V_5)(\tau) = (-1)^{k+1} F(\tau).$$ When $k$ is even this forces $F = 0$ along $\mathcal{H}_5$. \\

The claim regarding the order along $\mathcal{H}_5$ follows from the existence of a meromorphic paramodular form $b_5^2 / b_8$ with only a double pole along $\mathcal{H}_5$. If a form $F$ as in the claim vanishes to some even order $k$ along $\mathcal{H}_5$, then the product $F \cdot (b_5^2 / b_8)^{k/2}$ is holomorphic, antisymmetric, and by the previous paragraph must vanish along $\mathcal{H}_5$. But then $F$ must have vanished to order at least $k+1$. \\

The argument for $\mathcal{H}_{20}$ is similar. One shows that $F$ must be zero on $\mathcal{H}_{20}$ (see \cite{M}, section 7.2). Then if $F$ vanishes along $\mathcal{H}_{20}$ to some even order $k$ then $F \cdot (b_8 / b_{12})^k$ is holomorphic and also antisymmetric so it has a forced zero along $\mathcal{H}_{20}$, i.e. $\mathrm{ord}_{\mathcal{H}_{20}}(F) \ge 1 + k \cdot \mathrm{ord}_{\mathcal{H}_{20}}(b_{12}/b_8) = k+1$.
\end{proof}

\textbf{Symmetric even-weight paramodular forms.} In this section we will compute a system of generators for the ring of symmetric paramodular forms of even weight. This is done by induction on the weight. Every weight zero paramodular form is constant. In general we will find a family of generators containing $b_8$ such that, given a symmetric even-weight paramodular form $F$, some polynomial expression $P$ in those generators coincides with $F$ to order at least two along $\mathcal{H}_4$ and $\mathcal{H}_5$ and to order at least four along the diagonal $\mathcal{H}_1$. The quotient $\frac{F - P}{b_8}$ is then holomorphic and has smaller weight than $F$, so by induction $\frac{F - P}{b_8}$ and therefore also $F$ is a polynomial expression in those generators.

\begin{lem} Let $F$ be a symmetric even-weight paramodular form. There is a polynomial $P$ such that $F - P(\mathcal{E}_4,\mathcal{E}_6,g_6,g_8,g_{10})$ has at least a double zero along $\mathcal{H}_4$.
\end{lem}
\begin{proof} The point is that the images of $\mathcal{E}_4,\mathcal{E}_6,g_6,g_8,g_{10}$ under the pullback $P_4$ generate all possible pullbacks. If $F(\begin{psmallmatrix} \tau & z \\ z & w \end{psmallmatrix}) = \sum_{a,b,c} \alpha(a,b,c) q^a r^b s^{5c}$ is a paramodular form then \begin{align*} P_4 F(\tau_1,\tau_2) &= F( \begin{psmallmatrix} 2\tau_1 & \tau_1 \\ \tau_1 & \tau_1/2 + \tau_2/10 \end{psmallmatrix}) \\ &= \Big( \alpha(0,0,0) + \alpha(1,0,0)q_1^2 + O(q_1^4) \Big) \\ &\quad + \Big( (\alpha(1,-4,1) + \alpha(2,-6,1))q_1^{1/2} + (\alpha(1,-3,1) + \alpha(2,-5,1) + \alpha(3,-7,1)) q_1^{3/2} + O(q_1^{5/2}) \Big) q_2^{1/2} \\ &\quad + O(q_2). \end{align*} In particular, if $F$ is the Gritsenko lift of a Jacobi form $\phi(\tau,z) = \sum_{n,r} c(n,r) q^n \zeta^r \in J_{k,5}$, then the Maass relations allow us to simplify this to $$P_4 F(\tau_1,\tau_2) = -\frac{B_k}{2k} c(0,0) \Big( 1 - \frac{2k}{B_k} q_1^2 + O(q_1^4) \Big) + (-1)^k \Big( 2c(1,4) q_1^{1/2} + (2c(1,3) + c(2,5))q_1^{3/2} + O(q_1^{5/2}) \Big) q_2^{1/2} + O(q_2).$$ In this way we compute \begin{align*} P_4 \mathcal{E}_4(\tau_1,\tau_2) &= (1 + 240q_1^2 + ...) + (480q_1^{1/2} + 13440q_1^{3/2} + ...)q_2^{1/2} + O(q_2); \\ P_4 \mathcal{E}_6(\tau_1,\tau_2) &= (1 - 504q_1^2 - ...) + (55440/521 q_1^{1/2} + 16140096/521 q_1^{3/2} + ...)q_2^{1/2} + O(q_2); \\ P_4 g_6(\tau_1,\tau_2) &= (2q_1^{1/2} - 24q_1^{3/2} + ...)q_2^{1/2} + O(q_2); \\ P_4 g_8(\tau_1,\tau_2) &= (2q_1^{1/2} + 24q_1^{3/2} + ...)q_2^{1/2} + O(q_2); \\ P_4 g_{10}(\tau_1,\tau_2) &= (16q_1^{3/2} + ...)q_2^{1/2} + O(q_2), \end{align*} and comparing these coefficients with the generators of $A_*^{sym}$ found in Proposition~\ref{genA} allows us to identify $$P_4 \mathcal{E}_4 = X_2^2 - 48X_4, \; P_4 \mathcal{E}_6 = X_2^3 - 72X_2 X_4 - \frac{319680}{521} \Delta_6, \; P_4 g_6 = 2\Delta_6, \; P_4 g_8 = 2X_2 \Delta_6, \; P_4 g_{10} = 16 X_4 \Delta_6.$$ It is clear that these forms also generate $A_*^{sym}$. \\

Since $P_4 F \in A_*^{sym}$, it follows that there is some polynomial $P$ for which $$P_4 F = P(P_4 \mathcal{E}_4,P_4 \mathcal{E}_6,P_4 g_6,P_4 g_8,P_4 g_{10}) = P_4 P(\mathcal{E}_4,\mathcal{E}_6,g_6,g_8,g_{10}).$$ The difference $F - P(\mathcal{E}_4,\mathcal{E}_6,g_6,g_8,g_{10})$ has a zero along $\mathcal{H}_4$; and as we argued in Lemma~\ref{level5orders}, every symmetric even-weight form has even order on $\mathcal{H}_4$ so this must be at least a double zero.
\end{proof}

\begin{lem} Let $F$ be a symmetric even-weight paramodular form. There is a polynomial $P$ such that $$F - P(\mathcal{E}_4,\mathcal{E}_6,g_6,g_8,b_5^2,g_{10},b_5 g_7)$$ has double zeros on both $\mathcal{H}_4$ and $\mathcal{H}_5$.
\end{lem}
\begin{proof} By the previous lemma we can assume without loss of generality that $F$ vanishes along $\mathcal{H}_4$. Since $\mathcal{H}_4 \cap \mathcal{H}_5$ is the diagonal within $\mathcal{H}_5$ (interpreted as the Hilbert modular surface for $\mathbb{Q}(\sqrt{5})$) and since $\mathcal{H}_4$ contains the one-dimensional cusps, it follows that $P_5 F$ is a Hilbert cusp form for $\mathrm{SL}_2(\mathcal{O}_K)$ which is zero along the diagonal. From Gundlach's structure theorem we see that $P_5 F \in s_5^2 \mathbb{C}[\mathbf{E}_2,\mathbf{E}_6,s_5^2].$ \\

It turns out that $\mathbf{E}_2$ does not arise as a pullback. (In fact the table of Borcherds products (Table 1) immediately shows that there are no nonzero paramodular forms $F$ of weight two. One can assume that $F$ is symmetric or anti-symmetric by splitting into parts. The quasi-pullback of $F$ to the diagonal would be a cusp form of weight at least 12, so $\mathrm{ord}_{\mathcal{H}_1} F \ge 10$. Therefore the quasi-pullback of $F$ to $\mathcal{H}_4$ would also be a cusp form so it would have weight at least 6, so $\mathrm{ord}_{\mathcal{H}_4} F \ge 4$. But then $F / b_5$ would have negative weight and be holomorphic by Koecher's principle, which implies that it is zero. The same argument also gives $S_4(K(5)) = 0$.) Instead we decompose $$s_5^2 \mathbb{C}[\mathbf{E}_2,\mathbf{E}_6,s_5^2] = s_5^2 \mathbb{C}[\mathbf{E}_2^2,\mathbf{E}_6,s_5^2] \oplus \mathbf{E}_2 s_5^2 \mathbb{C}[\mathbf{E}_2^2,\mathbf{E}_6,s_5^2]$$ by sorting monomials according to whether the exponent of $\mathbf{E}_2$ in them is even or odd, and are left to find paramodular forms of weights $4,6,10,12$ which pullback to $\mathbf{E}_2^2,\mathbf{E}_6,s_5^2,\mathbf{E}_2 s_5^2$. We can compute $$P_5 \mathcal{E}_4 = \mathbf{E}_2^2,\; P_5 b_5 = s_5, \; P_5\Big(\mathcal{E}_6 - \frac{544320}{34907}g_6\Big) = \mathbf{E}_6,\; P_5 g_7 = \mathbf{E}_2 s_5,$$ (all of which except the third require no computation as the target space of Hilbert modular forms or cusp forms is one-dimensional), so there are polynomials $P_1,P_2$ such that $$F - b_5^2 P_1(\mathcal{E}_4,\mathcal{E}_6,g_6,b_5^2,b_5 g_7) - b_5 g_7 P_2(\mathcal{E}_4,\mathcal{E}_6,g_6,b_5^2,b_5g_7)$$ is zero along $\mathcal{H}_5$. Since its order along $\mathcal{H}_5$ is even (by Lemma~\ref{level5orders}), this must be at least a double zero. Moreover, this expression continues to have a double zero along $\mathcal{H}_4$ since $b_5^2$ and $b_5 g_7$ do (indeed, $b_5$ and $g_7$ have odd weight and therefore forced zeros on $\mathcal{H}_4$).
\end{proof}

\begin{lem}\label{symmetriclvl5} Define $h_{10} = \frac{g_7 b_8}{b_5}$ and $h_{12} = \frac{g_7^2 b_8}{b_5^2}$. Every symmetric even-weight paramodular form is a polynomial expression in the generators $$\mathcal{E}_4,\mathcal{E}_6,g_6,b_8,g_8,b_5^2,g_{10},h_{10},b_5 g_7, h_{12}.$$
\end{lem}
\begin{proof} Induction on the weight. This is trivial when the weight is negative or zero. \\

Let $F$ be a symmetric even-weight paramodular form. By the previous lemmas we can assume without loss of generality that $F$ has vanishes to order at least two on $\mathcal{H}_4$ and $\mathcal{H}_5$. To reduce by the Borcherds product $b_8$ with $\mathrm{div} \, b_8 = 4\mathcal{H}_1 + 2\mathcal{H}_4 + 2\mathcal{H}_5$ we need to subtract off expressions from $F$ which also vanish along $\mathcal{H}_4$ and $\mathcal{H}_5$ and which eliminate the cases that $F$ is nonzero or has only a double zero along the diagonal. \\

The pullback $P_1 F$ to $\mathcal{H}_1$ is a cusp form for $\mathrm{SL}_2(\mathbb{Z}) \times \mathrm{SL}_2(\mathbb{Z})$. Using the fact that $M_*(\mathrm{SL}_2(\mathbb{Z}))$ is the polynomial ring $\mathbb{C}[E_4,E_6]$ it is not difficult to see that the ring of symmetric modular forms for $\mathrm{SL}_2(\mathbb{Z}) \times \mathrm{SL}_2(\mathbb{Z})$ is $$M^{sym}_*(\mathrm{SL}_2(\mathbb{Z}) \times \mathrm{SL}_2(\mathbb{Z})) = \mathbb{C}[E_4 \otimes E_4, E_6 \otimes E_6, \Delta \otimes \Delta],$$ where we denote $f \otimes f(\tau_1,\tau_2) = f(\tau_1)f(\tau_2)$. Here $E_4,E_6,\Delta = \frac{E_4^3 - E_6^2}{12^3}$ are the level one Eisenstein series and discriminant. The ideal of cusp forms is generated by $\Delta \otimes \Delta$. Comparing constant terms shows that $P_1 \mathcal{E}_4 = E_4 \otimes E_4$ and $P_1 \mathcal{E}_6 = E_6 \otimes E_6$, so our task is to find paramodular forms in weights 10 and 12 with double zeros along $\mathcal{H}_4$ and $\mathcal{H}_5$ whose (quasi-)pullback to $\mathcal{H}_1$ is $\Delta \otimes \Delta$. Since the space of weight 12 cusp forms for $\mathrm{SL}_2(\mathbb{Z}) \times \mathrm{SL}_2(\mathbb{Z})$ is one-dimensional, it is enough to produce any paramodular forms in weights 10 and 12 with orders along $\mathcal{H}_1$ exactly 2 and 0, respectively. \\

The quotients $h_{10}$ and $h_{12}$ have this property. First note that they are holomorphic, since $g_7$ has a zero on $\mathcal{H}_4$ (due to its odd weight) and at least a fifth-order zero on the diagonal (since its quasi-pullback is a cusp form and therefore divisible by $\Delta \otimes \Delta$, which has weight 12). They have (at least) double zeros along $\mathcal{H}_4$ and $\mathcal{H}_5$ due to the $b_8$ in their numerators. To prove that $\mathrm{ord}_{\mathcal{H}_1} h_{10} = 2$ and $\mathrm{ord}_{\mathcal{H}_1} h_{12} = 0$ we need to show that $g_7$ has order \emph{exactly} five on the diagonal. But if $\mathrm{ord}_{\mathcal{H}_1} g_7 \ge 7$, then the quotient $\frac{g_7}{b_5}$ would be a holomorphic paramodular form of weight two. Considering the possible weights of its quasi-pullbacks to $\mathcal{H}_1$ and $\mathcal{H}_4$ shows that $\mathrm{ord}_{\mathcal{H}_1} (g_7/b_5) \ge 10$ and $\mathrm{ord}_{\mathcal{H}_4}(g_7/b_5) \ge 4$, and in particular that $(g_7 / b_5)$ is again divisible by $b_5$. This is a contradiction as $g_7 / b_5^2$ has negative weight. \\

In particular, there are polynomials $P_1$ and $P_2$ such that $F$ and $h_{12} P_1(\mathcal{E}_4,\mathcal{E}_6,h_{12})$ have the same pullback to $\mathcal{H}_1$ and such that $F-h_{12}P_1(\mathcal{E}_4,\mathcal{E}_6,h_{12})$ and $h_{10}P_2(\mathcal{E}_4,\mathcal{E}_6,h_{12})$ have the same quasi-pullback. The quotient $$\frac{F - h_{12} P_1(\mathcal{E}_4,\mathcal{E}_6,h_{12}) - h_{10} P_2(\mathcal{E}_4,\mathcal{E}_6,h_{12})}{b_8}$$ is then holomorphic and of smaller weight, so the claim follows.
\end{proof}

\textbf{Antisymmetric even-weight paramodular forms.} We will compute the ring of even-weight paramodular forms by reducing against the antisymmetric Borcherds product $b_{12}$ which has weight $12$, trivial character, and divisor $\mathrm{div} \, b_{12} = 3 \mathcal{H}_5 + \mathcal{H}_{20}.$ Every antisymmetric paramodular form of even weight vanishes on the Humbert surface $\mathcal{H}_{20}$ by Lemma~\ref{level5orders} so we only need to consider quasi-pullbacks to the Humbert surface $\mathcal{H}_5$ to account for the rest of the divisor of $b_{12}$. \\

The quasi-pullback of any antisymmetric, even-weight paramodular form $F$ to $\mathcal{H}_5$ is a Hilbert modular form for $\mathbb{Q}(\sqrt{5})$ which is antisymmetric if its weight is even (or equivalently, if the order of $F$ on $\mathcal{H}_5$ is even) and symmetric if its weight is odd. From Gundlach's structure theorem one can infer that any such Hilbert modular form is a multiple of the form $s_{15}$ in odd weight, and $s_5 s_{15}$ in even weight. Organizing monomials by the power of $\mathbf{E}_2$ they contain shows that the space of such Hilbert modular forms is exactly $$s_{15} \mathbb{C}[\mathbf{E}_2^2,\mathbf{E}_6,s_5^2] \oplus \mathbf{E}_2 s_{15} \mathbb{C}[\mathbf{E}_2^2,\mathbf{E}_6,s_5^2].$$ The problem is then to produce paramodular forms of weights $14$ and $16$ with simple zeros along $\mathcal{H}_5$ whose quasi-pullbacks to $\mathcal{H}_5$ are $s_{15}$ and $\mathbf{E}_2s_{15}$. We get the following proposition.

\begin{lem}\label{evenlvl5} Let $h_{16} = \frac{b_5 g_7 b_{12}}{b_8}$. Every even-weight paramodular form is a polynomial expression in the generators $$\mathcal{E}_4,\mathcal{E}_6,g_6,g_8,b_4^2,b_5^2,g_{10},h_{10},b_5g_7,h_{12},b_{12},b_{14},h_{16}$$ of weights $4,6,6,8,8,10,10,12,12,12,14,16$.
\end{lem}
\begin{proof} Any paramodular form $F$ can be split into its symmetric and antisymmetric parts as $$F = \frac{F + F|V_5}{2} + \frac{F - F|V_5}{2}.$$ The symmetric part of $F$ is accounted for by the previous section. Therefore we assume without loss of generality that $F$ is antisymmetric under $V_5$. \\

Lemma~\ref{level5orders} shows that $F$ has odd order on $\mathcal{H}_5$. Suppose that order is one. Then the quasi-pullback of $F$ is a symmetric Hilbert modular form of odd weight and therefore takes the form $$s_{15} P_1(\mathbf{E}_2^2,\mathbf{E}_6,s_5^2,\mathbf{E}_2 s_5^2) + \mathbf{E}_2 s_{15} P_2(\mathbf{E}_2^2,\mathbf{E}_6,s_5^2,\mathbf{E}_2 s_5^2)$$ for some polynomials $P_1,P_2$. It is enough to produce any antisymmetric paramodular forms of weight $14$ and $16$ with only simple zeros along $\mathcal{H}_5$, since their quasi-pullbacks can only be $s_{15}$ and $\mathbf{E}_2 s_{15}$ up to a nonzero scalar multiple. For this we can take the holomorphic quotients $b_{14}$ and $h_{16}$ in the claim. Then $$F- b_{14} P_1(\mathcal{E}_4,\mathcal{E}_6 - Cg_6,b_5^2,b_5 g_7) - h_{16} P_2(\mathcal{E}_4,\mathcal{E}_6 - Cg_6,b_5^2, b_5g_7)$$ vanishes to order at least two (and therefore at least three) along $\mathcal{H}_5$. In particular it is divisible by $b_{12}$, with the quotient being a symmetric, even-weight paramodular form, so the claim follows from the symmetric case (Lemma~\ref{symmetriclvl5}).
\end{proof}

\textbf{Odd-weight paramodular forms.} We retain the notation from the previous subsections. The strategy to handle odd-weight paramodular forms will be reduction against the Borcherds product $b_5$. Recall that $\mathrm{div} \, b_5 = 7\mathcal{H}_1 + \mathcal{H}_4$. \\

\noindent \textbf{Theorem 1.} \emph{ Let $h_9 = \frac{g_6 b_8}{b_5}$, $h_{10} = \frac{g_7 b_8}{b_5}$, $h_{11} = \frac{g_6 g_7 b_8}{b_5^2}$, $h_{12} = \frac{g_7^2 b_8}{b_5^2}$, $h_{16} = \frac{b_5 g_7 b_{12}}{b_8}$. Every paramodular form of level 5 is an isobaric polynomial in the generators $$\mathcal{E}_4,b_5,\mathcal{E}_6,g_6,g_7,g_8,b_8,h_9,g_{10},h_{10},h_{11},b_{12},h_{12},b_{14},h_{16}.$$ The graded ring $M_*(K(5))$ is minimally presented by these generators and by $59$ relations in weights $13$ through $32$.}
\begin{proof} Let $F$ be a paramodular form of odd weight. Then the quasi-pullback of $F$ to $\mathcal{H}_1$ is a cusp form for $\mathrm{SL}_2(\mathbb{Z}) \otimes \mathrm{SL}_2(\mathbb{Z})$ and therefore some expression in $$(\Delta \otimes \Delta) \cdot \mathbb{C}[E_4 \otimes E_4,E_6 \otimes E_6, \Delta \otimes \Delta].$$ Since $\mathcal{E}_4,\mathcal{E}_6,h_{12}$ pullback to $E_4 \otimes E_4$, $E_6 \otimes E_6$ and $\Delta \otimes \Delta$, it is enough to produce paramodular forms in weights $11$, $9$ and $7$ whose quasi-pullbacks to $\mathcal{H}_1$ are $\Delta \otimes \Delta$. \\

We claim that $h_{11},h_9$ and $g_7$ have this property (at least up to a nonzero scalar multiple). Since the weight 12 cusp space for $\mathrm{SL}_2(\mathbb{Z}) \otimes \mathrm{SL}_2(\mathbb{Z})$ is one-dimensional it is enough to verify that $g_7,h_9,h_{11}$ vanish to order exactly $5,3,1$ respectively along $\mathcal{H}_1$. This was checked for $g_7$ in the proof of Lemma~\ref{symmetriclvl5}. To prove this for $h_9$ and $h_{11}$ it is enough to show that $g_6$ vanishes to order exactly six along $\mathcal{H}_1$. \\

Suppose $\mathrm{ord}_{\mathcal{H}_1} g_6 > 6$. Then the quasi-pullback of $g_6$ to $\mathcal{H}_1$ is a cusp form of weight greater than $12$. Since $\mathrm{SL}_2(\mathbb{Z}) \times \mathrm{SL}_2(\mathbb{Z})$ admits no cusp forms of weight $14$ other than zero, it follows that the quasi-pullback has weight at least $16$ and therefore $\mathrm{ord}_{\mathcal{H}_1} g_6 \ge 10$. In particular we find $$\mathrm{ord}_{\mathcal{H}_1} (g_6 b_8) \ge 14, \; \mathrm{ord}_{\mathcal{H}_4} (g_6 b_8) = 2,$$ so the quotient $g_6 b_8 / b_5^2$ is holomorphic of weight four. It vanishes on $\mathcal{H}_5$ and therefore its image under $P_4$ vanishes along the diagonal, so must be zero. This implies that $g_6 b_8 / b_5^2$ is a cusp form. Considering the possible weights of its quasi-pullback to $\mathcal{H}_1$ shows that $\mathrm{ord}(g_6 b_8 / b_5^2) \ge 8$. In particular the quotient $g_6 b_8 / b_5^3$ is holomorphic; but this is a contradiction, as it has negative weight $-1$. \\

It follows from this that there are polynomials $P_1,P_2,P_3$ such that $$F - h_{11} P_1(\mathcal{E}_4,\mathcal{E}_6,h_{12}) - h_9 P_2(\mathcal{E}_4,\mathcal{E}_6,h_{12}) - g_7 P_3(\mathcal{E}_4,\mathcal{E}_6,h_{12})$$ vanishes to order at least $7$ along the diagonal. It also vanishes along $\mathcal{H}_4$ because its weight is odd. Therefore, it is divisible by the Borcherds product $b_5$ with the quotient having even weight, so the claim follows from Lemma~\ref{evenlvl5}. \\

To complete the proof we need to compute the ideal of relations. For this it is helpful to know the dimensions $\mathrm{dim}\, M_k(K(5))$. This was worked out by Marschner (\cite{M}, Corollary 7.3.4) based on Ibukiyama's \cite{I} formula for $\mathrm{dim}\, S_k(p)$, $p$ prime, $k \ge 5$: \begin{align*} \mathrm{Hilb} \, M_*(K(5)) &= \sum_{k=0}^{\infty} \mathrm{dim}\, M_k(K(5)) t^k \\ &= \frac{P(t)}{(1 - t^4)(1 - t^5)(1 - t^6)(1 - t^{12})} \end{align*} with $P(t) = 1 + t^6 + t^7 + 2t^8 + t^9 + 2t^{10} + t^{11} + 2t^{12} + 2t^{14} + 2t^{16} + 2t^{18} + t^{19} + 2t^{20} + t^{21} + 2t^{22} + t^{23} + t^{24} + t^{30}$. \\ Since we are given spanning sets of paramodular forms, we only need to compute their Fourier expansions up to a precision sufficient to find pivot coefficients in all necessary weights, and find enough relations among them to cut the dimension down to the correct value. There are effective upper bounds on the necessary precision (for example \cite{BPY} for Fourier-Jacobi expansions of paramodular forms of degree two, or \cite{PY} in general). But one can guess the correct result: since the form $g_5^k$ of weight $5k$ can only be distinguished from zero by its first $k$ Fourier-Jacobi coefficients, we might expect to need $\lceil 32/5 \rceil = 7$ Fourier-Jacobi coefficients to determine all relations up to weight $32$. This turns out to be enough. Finally we checked that the ideal generated by relations of weight up to 32 yields the Hilbert series predicted by Ibukiyama's formula. \end{proof}

It was conjectured in \cite{M} that $M_*(K(5))$ is Cohen-Macaulay. This follows from the computation above.

\begin{cor} The graded ring $M_*(K(5))$ is a Gorenstein ring which is not a complete intersection.
\end{cor}
\begin{proof} In principle, one can test algorithmically whether any graded ring given by explicit generators and relations is Cohen-Macaulay (for example, using Algorithm 5.2 of \cite{SW}). However, with so many generators and relations it is far easier to guess a sequence of four modular forms (here $\mathcal{E}_4,b_5,\mathcal{E}_6,b_{12}$) and verify that this is a homogeneous system of parameters and a $M_*(K(5))$-sequence (which can be done quickly in Macaulay2). For these notions and their relation to the Cohen-Macaulay property we refer to (for example) chapter 6 of \cite{BG}, especially Proposition 6.7. \\

By Stanley (\cite{S}, Corollary 3.3 and Theorem 4.4) the claim can be read off of the Hilbert polynomial $$P(t) = (1 - t^4)(1-t^5)(1-t^6)(1-t^{12}) \cdot \mathrm{Hilb} \, M_*(K(5)).$$ $M_*(K(5))$ is Gorenstein because $P(t)$ is palindromic and it is not a complete intersection because $P(t)$ does not factor into cyclotomic polynomials.
\end{proof}

(The published version of this note contains some incorrect remarks regarding the field of paramodular functions and the rationality of $X_{K(5)}$. Whether $X_{K(5)}$ is a rational variety seems to be open. I thank G. Sankaran for pointing this out to me.)

\newpage

\section{Paramodular forms of level 7}

In this section we compute generators for $M_*(K(7))$ in terms of the Borcherds products and Gritsenko lifts in tables 3 and 4 below. The procedure is roughly the same as what we used for level $5$.

\begin{longtable}{|l|l|l|l|}
\caption{Some Borcherds products of level 7} \\
\hline
   Name  &    Weight       &  Symmetric &  Divisor                                                       \\ \hline
       $b_4$ & $4$ & yes & $8\mathcal{H}_1 + 2 \mathcal{H}_4$ \\ \hline
       $b_6$ & $6$  & yes & $6\mathcal{H}_1 + \mathcal{H}_9$ \\ \hline
       $b_7$ & $7$  & no & $5\mathcal{H}_1 + 3 \mathcal{H}_4 + \mathcal{H}_8$ \\ \hline
       $b_9$ & $9$  & no & $3 \mathcal{H}_1 + \mathcal{H}_4 + \mathcal{H}_8 + \mathcal{H}_9$ \\ \hline
       $b_{10}$ & $10$ &  yes & $2\mathcal{H}_1 + 4\mathcal{H}_4 + \mathcal{H}_{21}$ \\ \hline
       $b_{12}^{sym}$ & $12$ & yes & $2\mathcal{H}_4 + \mathcal{H}_9 + \mathcal{H}_{21}$ \\ \hline
       $b_{12}^{anti}$ & $12$ & no & $\mathcal{H}_8 + \mathcal{H}_{28}$ \\ \hline
       $b_{13}$ & $13$ & yes & $11\mathcal{H}_1 + \mathcal{H}_4 + \mathcal{H}_{28}$ \\ \hline
\end{longtable}

Tables of Borcherds products (including those with character) appear in section 7.3 of \cite{G} and appendix A of \cite{W}. In the notation of \cite{W} these are the forms $b_4 = \psi_2^2$, $b_6 = \psi_6$, $b_7 = \psi_2 \psi_5$, $b_9 = \frac{b_6b_7}{b_4} = \psi_2^{-1} \psi_5 \psi_6$, $b_{10} = \psi_{10}^{(1)}$, $b_{12}^{sym} = \frac{b_6 b_{10}}{b_4} = \psi_2^{-2} \psi_6 \psi_{10}^{(1)}$, $b_{12}^{anti} = \psi_2^{-2} \psi_5 \psi_{11}$ and $b_{13} = \frac{b_4^2 b_{12}^{anti}}{b_7} = \psi_2 \psi_{11}$.

We will also need the following Gritsenko lifts.

\begin{longtable}{|l|l|l|}
\caption{Some Gritsenko lifts of level 7} \\
\hline
   Name  &    Weight &    Input Jacobi form                                                      \\ \hline
	$g_5$ & $5$ & $\begin{aligned} &\quad \Big( (\zeta^{-5} - \zeta^5) - 2(\zeta^{-4} - \zeta^4) - 9(\zeta^{-3} - \zeta^3) + 36(\zeta^{-2} - \zeta^2) - 42(\zeta^{-1} - \zeta)\Big)q \\ &+ \Big( -18(\zeta^{-6} - \zeta^6) + 72(\zeta^{-5} - \zeta^5) - 72(\zeta^{-4} - \zeta^4) - 72(\zeta^{-3} - \zeta^3) \\ &\quad\quad + 198(\zeta^{-2} - \zeta^2) - 144(\zeta^{-1} - \zeta) \Big)q^2 + O(q^3) \end{aligned}$ \\ \hline
       $g_6$ & $6$ & $\begin{aligned} &\quad \Big( (\zeta^{-5} + \zeta^5) + 6 (\zeta^{-4} + \zeta^4) - 35(\zeta^{-3} + \zeta^3) + 40(\zeta^{-2} + \zeta) + 34(\zeta^{-1} + \zeta) - 92 \Big) q \\ &+ \Big( -2(\zeta^{-7} + \zeta^7) - 6(\zeta^{-6} + \zeta^6) -70(\zeta^{-5} + \zeta^5) + 388(\zeta^{-4} + \zeta^4) - 546(\zeta^{-3} + \zeta^3) \\ &\quad\quad + 38(\zeta^{-2} + \zeta^2) + 618(\zeta^{-1} + \zeta) -840 \Big)q^2 + O(q^3) \end{aligned}$ \\ \hline
	$g_7$ & $7$ & $\begin{aligned} &\quad \Big( (\zeta^{-5} - \zeta^5) + 11(\zeta^{-3} - \zeta^3) - 64(\zeta^{-2} - \zeta^2) + 90(\zeta^{-1} - \zeta)\Big)q \\ &+ \Big( 24(\zeta^{-6} - \zeta^6) - 232(\zeta^{-5} - \zeta^5) + 352(\zeta^{-4} - \zeta^4) + 648(\zeta^{-3} - \zeta^3) \\ &\quad\quad  -2312(\zeta^{-2} - \zeta^2) + 2288(\zeta^{-1} - \zeta) \Big) q^2 + O(q^3) \end{aligned}$ \\ \hline
       $g_8$ & $8$ & $\begin{aligned} &\quad \Big( -5(\zeta^{-3} + \zeta^3) + 18(\zeta^{-2} + \zeta^2) - 27(\zeta^{-1} + \zeta) + 28 \Big) q \\ &+  \Big( (\zeta^{-7} + \zeta^7) + 6(\zeta^{-6} + \zeta^6) - 49(\zeta^{-5} + \zeta^5) + 166(\zeta^{-4} + \zeta^4) - 291(\zeta^{-3} + \zeta^3) \\ &\quad \quad + 338(\zeta^{-2} + \zeta^2) -429(\zeta^{-1} + \zeta) + 512 \Big)q^2 + O(q^3) \end{aligned}$ \\ \hline
       $g_{10}$ & $10$ & $\begin{aligned} &\quad \Big( (\zeta^{-2} + \zeta^2) - 2(\zeta^{-1} + \zeta) + 2 \Big)q \\ &+ \Big( -(\zeta^{-6} + \zeta^6) + 2(\zeta^{-5} + \zeta^5) - (\zeta^{-4} + \zeta^4) - 2(\zeta^{-3} + \zeta^3) \\ &\quad\quad - 15(\zeta^{-2} + \zeta^2) + 32(\zeta^{-1} + \zeta) - 30\Big)q^2 + O(q^3) \end{aligned}$ \\ \hline
\end{longtable}

We will need the structure of Hilbert modular forms for $K = \mathbb{Q}(\sqrt{2})$. This is a classical result (see Hammond, \cite{H}) and using the theory of Borcherds products the proof is quite short, so we recall the main ideas here. The symmetric, even-weight Hilbert modular forms are a polynomial ring $\mathbb{C}[\mathbf{E}_2,s_4,\mathbf{E}_6]$, where $\mathbf{E}_2,\mathbf{E}_6$ are Eisenstein series and where $s_4$ is a product of eight theta-constants whose divisor consists of a double zero along the diagonal (which can also be constructed as a Borcherds product). To get the full ring one observes that any antisymmetric, odd-weight Hilbert modular form has forced zeros on the rational quadratic divisors of discriminants $1$ and $4$; that any symmetric, odd-weight Hilbert modular form has forced zeros on the rational quadratic divisors of discriminants $2$ and $8$; and that any antisymmetric, even-weight Hilbert modular form has forced zeros on the rational quadratic divisors of discriminants $1,2,4,8$. Moreover one can construct an antisymmetric Borcherds product $s_5$ of weight $5$ and a symmetric Borcherds product $s_9$ of weight $9$, each of which has only simple zeros on the respective quadratic divisors. Thus by reducing against $s_5$, $s_9$ and their product we get the Hironaka decomposition $$M_*(\mathrm{SL}_2(\mathbb{Z}[\sqrt{2}])) = \mathbb{C}[\mathbf{E}_2,s_4,\mathbf{E}_6] \oplus s_5\mathbb{C}[\mathbf{E}_2,s_4,\mathbf{E}_6] \oplus s_9 \mathbb{C}[\mathbf{E}_2,s_4,\mathbf{E}_6] \oplus s_5 s_9 \mathbb{C}[\mathbf{E}_2,s_4,\mathbf{E}_6]$$ into symmetric and antisymmetric forms of even and odd weights. \\

\textbf{Symmetric even-weight paramodular forms.} The graded ring of symmetric even-weight paramodular forms of level $7$ will be computed by nearly the same argument that we used for level $5$. Here we reduce instead against the Borcherds product $b_4$ of weight $4$ with divisor $\mathrm{div} \, b_4 = 8 \mathcal{H}_1 + \mathcal{H}_4$.

\begin{lem} Let $F$ be a symmetric, even-weight paramodular form. Then there is a polynomial $P$ such that $F - P(\mathcal{E}_4,\mathcal{E}_6,b_6,g_8,g_{10})$ vanishes to order at least two along $\mathcal{H}_4$.
\end{lem}
\begin{proof} If $F( \begin{psmallmatrix} \tau & z \\ z & w \end{psmallmatrix}) = \sum_{a,b,c} \alpha(a,b,c) q^a r^b s^{7c}$ is a paramodular form of level $7$ then its image under $P_4$ begins \begin{align*} P_4 F(\tau_1,\tau_2) &= F \left( \begin{psmallmatrix} 2\tau_1 & \tau_1 \\ \tau_1 & \tau_1/2 + \tau_2/14 \end{psmallmatrix} \right) \\ &= \Big( \alpha(0,0,0) + \alpha(1,0,0) q_1^2 + O(q_1^4) \Big)  + \Big( (\alpha(1,-5,1) + \alpha(2,-7,1) + \alpha(3,-9,1) ) q_1^{1/2} \\ &\quad\quad\quad + ((\alpha(1,-4,1) + \alpha(2,-6,1) + \alpha(3,-8,1) + \alpha(4,-10,1)) q_1^{3/2} + O(q_1^{5/2}) \Big) q_2^{1/2} + O(q_2) \end{align*} If $F$ is the Gritsenko lift of a Jacobi form $\phi = \sum_{n,r} c(n,r) q^n \zeta^r \in J_{k,7}$ then this reduces to \begin{align*} P_4 F(\tau_1,\tau_2) &= -\frac{B_k}{2k} c(0,0) \Big( 1 - \frac{2k}{B_k} q_1^2 + O(q_1^4) \Big) \\ &\quad + (-1)^k \Big( (2c(1,5) + c(2,7)) q_1^{1/2} + 2 (c(1,4) + c(2,6)) q_1^{3/2} + O(q_1^{5/2}) \Big) q_2^{1/2} + O(q_2). \end{align*} In any case, pulling back the Eisenstein series and the forms $b_6,g_8,g_{10}$ yields \begin{align*} P_4 \mathcal{E}_4 &= (1 + 240q_1^2 + ...) + (480q_1^{1/2} + 13440q_1^{3/2} + ...)q_2^{1/2} + O(q_2) = X_2^2 - 48X_4; \\ P_4 \mathcal{E}_6 &= (1 - 504q_1^2 - ...) + (\frac{25200}{191} q_1^{1/2} + \frac{5858496}{191} q_1^{3/2} + ...)q_2^{1/2} + O(q_2) = X_2^3 - 72X_2 X_4 - \frac{112320}{191} \Delta_6; \\ P_4 b_6 &= (-q_1^{1/2}+ 12q_1^{3/2} + ...)q_2^{1/2} + O(q_2) = -\Delta_6; \\ P_4 g_8 &= (q_1^{1/2} + 12q_1^{3/2} + ...) q_2^{1/2} + O(q_2) = X_2 \Delta_6; \\ P_4 g_{10} &= (-2q_1^{3/2} + ...)q_2^{1/2} + O(q_2) = -2X_4 \Delta_6. \end{align*} It is clear that these forms generate the ring $A_*^{sym}$ of possible pullbacks by Proposition~\ref{genA}.
\end{proof}

\begin{lem}\label{symmetriclvl7} Let $h_8 = \frac{g_6^2}{b_4}$. The graded ring of symmetric even-weight paramodular forms is generated by $$\mathcal{E}_4,b_4,\mathcal{E}_6,b_6,g_6,g_8,h_8,b_{10},g_{10},b_{12}^{sym}.$$
\end{lem}
\begin{proof} Let $F$ be a symmetric even-weight paramodular form. By the previous lemma we assume without loss of generality that $F$ has at least a double zero along $\mathcal{H}_4$. In particular $F$ is a cusp form. To divide by $b_4$ we need to handle the possible cases that $F$ has order $0,2,4$ or $6$ along the diagonal. The argument is analogous to what we used for level $5$: since $$M_*^{sym}(\mathrm{SL}_2(\mathbb{Z}) \times \mathrm{SL}_2(\mathbb{Z})) = \mathbb{C}[E_4 \otimes E_4, E_6 \otimes E_6, \Delta \otimes \Delta]$$ and since $\mathcal{E}_4$ and $\mathcal{E}_6$ pullback to $E_4 \otimes E_4$ and $E_6 \otimes E_6$ along the diagonal, the task that remains is to find paramodular forms of weights $12,10,8,6$, each with at least a double zero along $\mathcal{H}_4$ and whose quasi-pullbacks to the diagonal are $\Delta \otimes \Delta$. \\

Since $\Delta \otimes \Delta$ is (up to scalar multiples) the unique cusp form of weight $12$, it is enough to produce any paramodular forms with the correct orders along $\mathcal{H}_1$ and $\mathcal{H}_4$. We claim that $b_{12}^{sym},b_{10},h_8,g_6$ have this property. The divisors of the Borcherds products $b_{12}^{sym}$ and $b_{10}$ can be read off of table 3 but we need to consider $h_8$ and $g_6$ more carefully. \\

First note that $g_6$ has a double zero along $\mathcal{H}_4$: fom the Fourier expansion of $P_4 F$ worked out in the previous lemma we see that $P_4 g_6(\tau_1,\tau_2) = 0 + O(q_2)$ and therefore $P_4 g_6 = 0$. Since $g_6$ is a cusp form, its pullback to $\mathcal{H}_1$ has weight at least 12, and therefore $\mathrm{ord}_{\mathcal{H}_1} g_6 \ge 6$. In particular, we find $\mathrm{ord}_{\mathcal{H}_1} h_8 \ge 4$ and $\mathrm{ord}_{\mathcal{H}_4} h_8 \ge 2$; and $\mathrm{ord}_{\mathcal{H}_1} h_8 = 4$ will follow from $\mathrm{ord}_{\mathcal{H}_1} g_6 = 6.$ \\

To prove this, suppose $g_6$ vanishes to order $\ge 8$; then $g_6 / b_4$ is holomorphic of weight two. Since $M_*(\mathrm{SL}_2(\mathbb{Z}) \times \mathrm{SL}_2(\mathbb{Z}))$ and $A_*$ do not admit modular forms of weight two, it follows that $g_6 / b_4$ vanishes along both $\mathcal{H}_1$ and $\mathcal{H}_4$. But considering the possible orders of its quasi-pullbacks shows that $\mathrm{ord}_{\mathcal{H}_1}(g_6 / b_4) \ge 10$ and $\mathrm{ord}_{\mathcal{H}_4}(g_6 / b_4) \ge 2$, so $(g_6 / b_4) / b_4$ is holomorphic of negative weight $(-2)$. This is a contradiction.
\end{proof}

\textbf{Antisymmetric even-weight paramodular forms.} We deal with antisymmetric forms by an argument similar to the level 5 case; namely, reduction against the Borcherds product $b_{12}^{anti}$. Recall $\mathrm{div} \, b_{12}^{anti} = \mathcal{H}_8 + \mathcal{H}_{28}$. Every antisymmetric, even-weight paramodular form of level $7$ has a forced zero on $\mathcal{H}_{28}$, so to reduce against $b_{12}^{anti}$ we only need to consider the possible pullbacks to $\mathcal{H}_8$. The result of such a pullback will be a Hilbert modular form for the field $K = \mathbb{Q}(\sqrt{2})$ of discriminant $8$:

\begin{lem} Fix the totally positive element $\lambda = 1 - \frac{1}{\sqrt{8}} \in \mathcal{O}_K^{\#}$ with conjugate $\lambda' = 1 + \frac{1}{\sqrt{8}}$. Every element of the Humbert surface $\mathcal{H}_8$ is equivalent under $K(7)^+$ to a matrix of the form $$\begin{psmallmatrix} \lambda \tau_1 + \lambda' \tau_2 & (\tau_1 + \tau_2)/2 \\ (\tau_1 + \tau_2)/2 & (2/7)\lambda' \tau_1 +(2/7)\lambda \tau_2 \end{psmallmatrix}, \; \; \tau_1,\tau_2 \in \mathbb{H}.$$ If $F$ is a paramodular form of weight $k$ then $$P_8 F(\tau_1,\tau_2) = F \Big( \begin{psmallmatrix} \lambda \tau_1 + \lambda' \tau_2 & (\tau_1 + \tau_2)/2 \\ (\tau_1 + \tau_2)/2 & (2/7)\lambda' \tau_1 +(2/7)\lambda \tau_2 \end{psmallmatrix} \Big)$$ is a Hilbert modular form of weight $k$ for $\mathrm{SL}_2(\mathbb{Z}[\sqrt{2}])$. The pullback $P_8$ preserves cusp forms and sends $V_7$-(anti)symmetric paramodular forms to (anti)symmetric Hilbert modular forms.
\end{lem}
\begin{proof} The proof is almost identical to Lemma~\ref{P5}. The only nontrivial point to check is that $P_8F$ transforms correctly under $(\tau_1,\tau_2) \mapsto (-1/\tau_1,-1/\tau_2)$, and this follows from the behavior of $F$ under $$R = \begin{psmallmatrix} 0 & 0 & 2 & 1 \\ 0 & 0 & 1 & 4/7 \\ -4 & 7 & 0 & 0 \\ 7 & -14 & 0 & 0 \end{psmallmatrix} \in K(7),$$ which maps $\begin{psmallmatrix} \lambda \tau_1 + \lambda' \tau_2 & (\tau_1 + \tau_2)/2 \\ (\tau_1 + \tau_2)/2 & 2\lambda' \tau_1/7 +2\lambda \tau_2/7 \end{psmallmatrix}$ to $\begin{psmallmatrix} \lambda (-\tau_1)^{-1} + \lambda' (-\tau_2)^{-1} & -(\tau_1^{-1} + \tau_2^{-1})/2 \\ -(\tau_1^{-1} + \tau_2^{-1})/2 & 2\lambda' (-\tau_1)^{-1}/7 +2\lambda (-\tau_2)^{-1}/7 \end{psmallmatrix}$.
\end{proof}

The antisymmetric even-weight Hilbert modular forms are exactly $s_5 s_9 \mathbb{C}[\mathbf{E}_2,s_4,\mathbf{E}_6]$. By separating monomials in $\mathbb{C}[\mathbf{E}_2,s_4,\mathbf{E}_6]$ into those containing an even or odd number of instances of $\mathbf{E}_2$, we see that for any  antisymmetric even-weight Hilbert modular form $f$ there are unique polynomials $P_1,P_2$ such that $$f = s_5 s_9 P_1(\mathbf{E}_2^2,s_4,\mathbf{E}_6) + s_5 s_9 \mathbf{E}_2 P_2(\mathbf{E}_2^2,s_4,\mathbf{E}_6).$$

\begin{lem}\label{evenlvl7} Define $h_{14} = \frac{g_5 b_{13}}{b_4}$ and $h_{16} = \frac{g_7 b_{13}}{b_4}$. The graded ring of even-weight paramodular forms is generated by the forms $$\mathcal{E}_4,b_4,\mathcal{E}_6,b_6,g_6,g_8,h_8,b_{10},g_{10},b_{12}^{sym},b_{12}^{anti},h_{14},h_{16}.$$
\end{lem}
\begin{proof} The problem is to produce antisymmetric paramodular forms of weights $14$ and $16$ whose images under $P_8$ are nonzero scalar multiples of $s_5 s_9$ and $\mathbf{E}_2 s_5 s_9$. For this it is enough to ensure that the images are nonzero. \\

Note that $h_{14}$ and $h_{16}$ are holomorphic since the zeros of $g_5$ and $g_7$ on $\mathcal{H}_4$ cancel out the zero of $b_4$. To prove that $P_8 h_{14}$ and $P_8 h_{16}$ are nonzero it is enough to show that $P_8 g_5$ and $P_8 g_7$ are nonzero. If $\phi(\tau,z) = \sum_{n,r} c(n,r) q^n \zeta^r$ is a Jacobi form of any weight and $\mathcal{G}_{\phi}$ is its Gritsenko lift, then a short calculation shows that the Fourier coefficient of $q_1^{1/2 + 1/\sqrt{8}} q_2^{1/2 - 1/\sqrt{8}}$ in $P_8 \mathcal{G}_{\phi}$ is just the coefficient $c(1,-5)$. This is enough to show that $P_8 g_5 = s_5$ and $P_8 g_7 = \mathbf{E}_2 s_5$. \\

Finally consider that $b_4$ and $b_6$ are cusp forms whose pullbacks to $\mathcal{H}_8$ are not zero so $$\mathbb{C}[\mathbf{E}_2^2, s_4, \mathbf{E}_6] \subseteq P_8 \Big( \mathbb{C}[\mathcal{E}_4,b_4,\mathcal{E}_6,b_6]\Big) .$$ In particular, if $F$ is any antisymmetric paramodular form of even weight then decomposing $$P_8 F = s_5 s_9 P_1(\mathbf{E}_2^2,s_4,\mathbf{E}_6) + \mathbf{E}_2 s_5 s_9 P_2(\mathbf{E}_2^2,s_4,\mathbf{E}_6) \in P_8 \mathbb{C}[\mathcal{E}_4,b_4,\mathcal{E}_6,b_6,h_{14},h_{16}]$$ shows that there is some polynomial $P$ for which $F - P(\mathcal{E}_4,b_4,\mathcal{E}_6,b_6,h_{14},h_{16})$ is antisymmetric and vanishes along $\mathcal{H}_8$ and is therefore divisible by $b_{12}^{anti}$. The quotient is symmetric and has even-weight so the claim follows from Lemma~\ref{symmetriclvl7}.
\end{proof}

\textbf{Odd-weight paramodular forms.} We finish the computation of the graded ring $M_*(K(7))$ by reducing all odd-weight paramodular forms against the Borcherds product $b_7$, by matching (quasi-)pullbacks successively to the Humbert surfaces $\mathcal{H}_8$, $\mathcal{H}_4$ and finally $\mathcal{H}_1$ against those of holomorphic quotients of Gritsenko lifts.

\begin{lem} Define $h_{15}^{(1)} = \frac{g_6 b_{13}}{b_4}.$ Let $F$ be an odd-weight paramodular form of level $7$. There is a polynomial $P$ such that $F - P(\mathcal{E}_4,b_4,g_5,\mathcal{E}_6,b_6,g_7,b_{13},h_{15}^{(1)})$ vanishes along $\mathcal{H}_8$.
\end{lem}
\begin{proof} Recall the decomposition of odd-weight Hilbert modular forms for the field $K = \mathbb{Q}(\sqrt{2})$: $$\bigoplus_{k \, \text{odd}} M_k(\mathrm{SL}_2(\mathbb{Z}[\sqrt{2}])) = s_5 \mathbb{C}[\mathbf{E}_2,s_4,\mathbf{E}_6] \oplus s_9 \mathbb{C}[\mathbf{E}_2,s_4,\mathbf{E}_6],$$ where $s_5$ and $s_9$ are the (unique) antisymmetric resp. symmetric Borcherds products with trivial character and weights $5$ and $9$. \\

Since $F$ has odd weight, it has a forced zero on the Humbert surface $\mathcal{H}_4$. Since $\mathcal{H}_4$ and $\mathcal{H}_8$ intersect along the diagonal it follows that $P_8 F$ vanishes along the diagonal. Therefore we cannot, for example, realize $s_9$ as the pullback of a paramodular form. Actually one can show as in \cite{H} that any symmetric form which vanishes on the diagonal is also divisible by Hammond's cusp form $s_4$, so the odd-weight Hilbert modular forms which vanish on the diagonal are exactly $s_4 s_9 \mathbb{C}[\mathbf{E}_2,s_4,\mathbf{E}_6]$. Since $\mathbf{E}_2$ does not arise as a pullback it is more convenient to write \begin{align*} P_8 F &\in  s_5 \mathbb{C}[\mathbf{E}_2,s_4,\mathbf{E}_6] \oplus s_4 s_9 \mathbb{C}[\mathbf{E}_2,s_4,\mathbf{E}_6] \\ &= s_5 \mathbb{C}[\mathbf{E}_2^2,s_4,\mathbf{E}_6] \oplus \mathbf{E}_2 s_5 \mathbb{C}[\mathbf{E}_2^2,s_4,\mathbf{E}_6] \oplus s_4 s_9 \mathbb{C}[\mathbf{E}_2^2,s_4,\mathbf{E}_6] \oplus \mathbf{E}_2 s_4 s_9 \mathbb{C}[\mathbf{E}_2^2,s_4,\mathbf{E}_6]. \end{align*}  

Therefore we need to find antisymmetric paramodular forms of weights $5$ and $7$ and symmetric paramodular forms of weights $13$ and $15$ which are not identically zero along $\mathcal{H}_8$, since the spaces of possible pullbacks are one-dimensional and the pullbacks will have to equal $s_5$, $\mathbf{E}_2s_5$, $s_4 s_9$ and $\mathbf{E}_2 s_4 s_9$ up to nonzero constant multiples. We observed earlier that $g_5$ and $g_7$ pullback to $s_5$ and $\mathbf{E}_2 s_5$. Since the input Jacobi form to the Gritsenko lift $g_6$ has a nonzero Fourier coefficient of $q \zeta^{-5}$, the same argument shows that $P_8 g_6$ is nonzero. Finally $g_6$ has a zero along $\mathcal{H}_4$ so the quotient $h_{15}^{(1)}$ is holomorphic and nonzero on $\mathcal{H}_8$. 
\end{proof}

\begin{lem} Let $h_9^{(1)} = \frac{g_6g_7}{b_4}$, $h_{11}^{(1)} = \frac{g_6^2 g_7}{b_4^2}$, $h_{11}^{(2)} = \frac{b_7 g_8}{b_4}$, $h_{13} = \frac{g_5^2 b_7}{b_4}$, $h_{15}^{(2)} = \frac{g_5^2 g_6 b_7}{b_4^2}$. For any odd-weight paramodular form $F$ there is a polynomial $P$ such that $$F - P(\mathcal{E}_4,b_4,g_5,\mathcal{E}_6,b_6,g_6,g_7,g_8,h_9^{(1)},g_{10},h_{11}^{(1)},h_{11}^{(2)},b_{12}^{anti},b_{13},h_{13},h_{15}^{(1)},h_{15}^{(2)})$$ has a zero on $\mathcal{H}_8$ and at least a triple zero along $\mathcal{H}_4$.
\end{lem}
\begin{proof} We need to understand the possible quasi-pullbacks of (symmetric or antisymmetric) odd-weight paramodular forms $F$ which vanish along $\mathcal{H}_8$ and whose zero along $\mathcal{H}_4$ is only of order one. Suppose first that $F$ is antisymmetric. \\

Since $\mathcal{H}_4 \cap \mathcal{H}_8$ is the diagonal in $\mathcal{H}_4$, if $F$ is any paramodular form which vanishes along $\mathcal{H}_8$ then its quasi-pullback to $\mathcal{H}_4$ vanishes along the diagonal. (This implies, for example, that $b_7$ itself has quasi-pullback $X_4 \Delta_6$.) The ideal of $A_*^{sym}$ of cusp forms which vanish along the diagonal is $X_4 \Delta_6 \mathbb{C}[X_2,X_4,\Delta_6]$, which one can show is contained in $$X_4 \Delta_6 A_*^{sym} + X_2 X_4 \Delta_6 A_*^{sym} + X_2^2 X_4 \Delta_6 A_*^{sym} + X_2^3 X_4 \Delta_6 A_*^{sym}.$$ Therefore we only need to find antisymmetric forms of weights $9,11,13,15$ which vanish along $\mathcal{H}_8$ and whose quasi-pullbacks to $\mathcal{H}_4$ are scalar multiples of the generators $X_4 \Delta_6$, $X_2X_4 \Delta_6$, $X_2^2 X_4 \Delta_6$, $X_2^3 X_4 \Delta_6$. \\

We claim that $b_9, h_{11}^{(2)},h_{13},h_{15}^{(2)}$ have this property. This is clear for $b_9$. First note that $h_{11}^{(1)},h_{13},h_{15}^{(1)}$ are holomorphic since $\mathrm{ord}_{\mathcal{H}_1} g_5 = 7$, $\mathrm{ord}_{\mathcal{H}_1} g_6 = 6$ and $\mathrm{ord}_{\mathcal{H}_1} g_8 = 4$, and they vanish on $\mathcal{H}_8$ since $b_7$ does. Note that $g_6$ has a double zero on $\mathcal{H}_4$. (To prove this one can simply compute $P_4(g_6b_6 / b_4) \ne 0$.) The quasi-pullback of $g_6$ to $\mathcal{H}_4$ is therefore a multiple of $X_2 \Delta_6$. Now we only need to use the fact that $b_7$ and $b_4$ quasi-pullback to $X_4 \Delta_6$ and $\Delta_6$ and the fact that $P_4(g_8)$ is $X_2 \Delta_6$. \\

It follows that there are polynomials $P_1,P_2,P_3,P_4$ such that the quasi-pullback of $F$ to $\mathcal{H}_4$ equals that of $$b_9 P_1(\mathcal{E}_4,\mathcal{E}_6,b_6,g_8,g_{10}) + h_{11}^{(2)} P_2(\mathcal{E}_4,\mathcal{E}_6,b_6,g_8,g_{10}) + h_{13} P_3(\mathcal{E}_4,\mathcal{E}_6,b_6,g_8,g_{10}) + h_{15}^{(2)} P_4(\mathcal{E}_4,\mathcal{E}_6,b_6,g_8,g_{10}),$$ and the claim follows because the latter continues to vanish on $\mathcal{H}_8$. \\

Now suppose that $F$ is symmetric. Then it has a forced zero on $\mathcal{H}_{28}$ and therefore the quotient $\frac{b_6^2 F}{b_{13}}$ is holomorphic and vanishes on $\mathcal{H}_8$. In particular the quasi-pullback of $\frac{b_6^2 F}{b_{13}}$ to $\mathcal{H}_4$ is a multiple of $X_4$. This implies that the quasi-pullback of $F$ to $\mathcal{H}_4$ is a multiple of $X_4 \Delta_6 X_8$. We can use nearly the same argument as the previous case if we can find symmetric forms of weights $17,19,21,23$ which vanish along $\mathcal{H}_8$ and have a simple zero along $\mathcal{H}_4$, and whose quasi-pullbacks to $\mathcal{H}_4$ are $X_4 \Delta_6 X_8$, $X_2 X_4 \Delta_6 X_8$, $X_2^2 X_4 \Delta_6 X_8$ and $X_2^3 X_4 \Delta_6 X_8$, respectively. \\

We claim that the products $g_5 b_{12}^{anti}, g_7 b_{12}^{anti}, h_9^{(1)} b_{12}^{anti}$ and $h_{11}^{(1)} b_{12}^{anti}$ have this property. Since $b_{12}^{anti}$ vanishes on $\mathcal{H}_8$ and $P_4 b_{12}^{anti} = X_4X_8$ it is enough to show that $g_5$ and $g_7$ have only a simple zero on $\mathcal{H}_4$. But by direct computation one can show (for example) that $P_4(g_5g_7 / b_4) \ne 0$, which is enough. \end{proof}

\begin{lem} Let $h_9^{(2)} = \frac{g_6 b_7}{b_4}$ and $h_{11}^{(3)} = \frac{g_6^2 b_7}{b_4^2}$. Every paramodular form of level $7$ is a polynomial expression in the generators $$\mathcal{E}_4,b_4,g_5,\mathcal{E}_6,g_6,b_6,b_7,g_7,g_8,h_8,b_9,h_9^{(1)},h_9^{(2)},b_{10},g_{10},h_{11}^{(1)},h_{11}^{(2)},h_{11}^{(3)},b_{12}^{sym},b_{12}^{anti},b_{13},h_{13},h_{14},h_{15}^{(1)},h_{15}^{(2)},h_{16}.$$ 
\end{lem}
\begin{proof} We only need to consider the case that $F$ has odd weight. By the previous lemmas we can assume without loss of generality that $\mathrm{ord}_{\mathcal{H}_8} F \ge 1$ and $\mathrm{ord}_{\mathcal{H}_4} F \ge 3$. To reduce against $b_7$ we need to handle the cases that $\mathrm{ord}_{\mathcal{H}_1} F \in \{1,3\}$. Since the argument is familiar by now we mention only that we need to find paramodular forms of weights $9$ and $11$, with at least triple zeros along $\mathcal{H}_4$, zeros along $\mathcal{H}_8$, and orders exactly $3$ and $1$ along the diagonal (such that the quasi-pullbacks are both $\Delta \otimes \Delta$). But it is easy to see that $h_9^{(2)}$ and $h_{11}^{(3)}$ have this property using the fact that $\mathrm{ord}_{\mathcal{H}_1} g_6= 6$. \\

Having subtracted away the possible first-order and third-order terms of $F$ along the diagonal, we obtain a form which is divisible by $b_7$. The quotient $F / b_7$ has even weight so the claim follows from Lemma~\ref{evenlvl7}.
\end{proof}

At this point we have found too many generators: the forms $h_9^{(2)},h_{11}^{(2)},h_{11}^{(3)},h_{13},h_{15}^{(2)}$ turn out to be unnecessary. It seems difficult to prove this without resorting to Fourier expansions. A computation shows \begin{align*} h_9^{(2)} &=  \frac{1}{14} h_9^{(1)} + \frac{48}{7} b_9 - \frac{1224}{301} b_4g_5 - \frac{1}{14} \mathcal{E}_4 g_5; \\  h_{11}^{(2)} &= - \frac{1}{384} h_{11}^{(1)} + \frac{2925}{191} g_5 b_6 - \frac{5}{192} g_5 \mathcal{E}_6 - \frac{33681}{16426} b_4 g_7 - \frac{1124189}{32852} b_4 b_7 + \frac{11}{384} \mathcal{E}_4 g_7 + \frac{1}{96} \mathcal{E}_4 b_7; \\ h_{11}^{(3)} &= \frac{1}{8} h_{11}^{(1)} - \frac{28080}{191} g_5 b_6 + \frac{1}{4} g_5 \mathcal{E}_6 + \frac{114912}{8213} b_4 g_7 + \frac{2201148}{8213} b_4 b_7 - \frac{3}{8} \mathcal{E}_4 g_7 - \frac{3}{2} \mathcal{E}_4 b_7; \\ h_{13} &= \frac{1}{14} b_4 h_9^{(1)} + \frac{20}{7} b_4 b_9 - \frac{1224}{301} b_4^2 g_5 - \frac{1}{14} \mathcal{E}_4 b_4 g_5; \\ h_{15}^{(2)} &= \frac{3}{32} b_4 h_{11}^{(1)} - \frac{11700}{191} b_4 g_5 b_6 + \frac{5}{48} b_4 g_5 \mathcal{E}_6 + \frac{28398}{8213} b_4^2 g_7 + \frac{722325}{8213} b_4^2 b_7 - \frac{19}{96} \mathcal{E}_4 b_4 g_7 - \frac{25}{24} \mathcal{E}_4 b_4 b_7. \end{align*} Therefore we can abbreviate $h_9 = h_9^{(1)}, h_{11} = h_{11}^{(1)}, h_{15} = h_{15}^{(1)}.$ Finally we obtain theorem 2 from the introduction:\\

\noindent \textbf{Theorem 2.} \emph{Let $h_8 = \frac{g_6^2}{b_4}$, $h_9 = \frac{g_6 g_7}{b_4}$, $h_{11} = \frac{g_6^2 g_7}{b_4^2}$, $h_{14} = \frac{g_5 b_{13}}{b_4}$, $h_{15} = \frac{g_6 b_{13}}{b_4}$, $h_{16} = \frac{g_7 b_{13}}{b_4}$. The graded ring of paramodular forms of level $7$ is minimally presented by the generators $$\mathcal{E}_4,b_4,g_5,\mathcal{E}_6,b_6,g_6,b_7,g_7,g_8,h_8,b_9,h_9,b_{10},g_{10},h_{11},b_{12}^{sym},b_{12}^{anti},b_{13},h_{14},h_{15},h_{16}$$ of weights $4,4,5,6,6,6,7,7,8,8,9,9,10,10,11,12,12,13,14,15,16$ and by 144 relations in weights $10$ through $32$.}
\begin{proof} The relations are computed as in the proof of Theorem 1. Here we use the Hilbert series $$\mathrm{Hilb} \, M_*(K(7)) = \sum_{k=0}^{\infty} \mathrm{dim}\, M_k(K(7)) = \frac{P(t)}{(1 - t^4)^2 (1-t^6)(1 - t^{12})},$$ where $$P(t) = 1 + t^5 + 2(t^6 + t^7 + ... + t^{22} + t^{23}) + t^{24} + t^{29},$$ which can be derived from Ibukiyama's formula \cite{I} for $\mathrm{dim}\, S_k(K(p))$, $k \ge 5$.
\end{proof}

\begin{cor} $M_*(K(7))$ is a Gorenstein graded ring which is not a complete intersection.
\end{cor}
\begin{proof} This is essentially the same as the proof for $M_*(K(5))$. One can guess the sequence of modular forms $\mathcal{E}_4,b_4,\mathcal{E}_6,b_{12}^{anti}$ (which is suggested by the denominator of the Hilbert series) and verify that it is a homogeneous system of parameters and a $M_*(K(7))$-sequence. (Verifying this is much faster in Macaulay2 than computing even the Krull dimension of $M_*(K(7))$ naively.) The rest of the claim can be read off of the Hilbert polynomial $P(t)$ by \cite{S}.
\end{proof}

\bibliographystyle{plainnat}
\bibliofont
\bibliography{\jobname}

\end{document}